\documentclass[11pt,reqno,a4paper]{amsart}
\usepackage[british]{babel}

\usepackage{multicol}
\usepackage{graphicx}
\usepackage{amscd}
\usepackage{amsmath}
\usepackage{amsfonts}
\usepackage{amssymb}
\usepackage{color}
\usepackage{pdfsync}
\usepackage{bbm}
\usepackage{dsfont}
\usepackage{graphicx}
\usepackage{hyperref}
\usepackage{amsaddr}

\definecolor{trp}{rgb}{1,1,1}

\definecolor{red}{rgb}{1,0,.2}

\newtheorem{theorem}{Theorem}[section]
\theoremstyle{plain}

\newtheorem{cor}[theorem]{Corollary}

\newtheorem{definition}[theorem]{Definition}

\newtheorem{lemma}[theorem]{Lemma}

\newtheorem{prop}[theorem]{Proposition}

\numberwithin{equation}{section}

\newcommand{\ly}[1]{\chi^{#1}_{\mu}}
\newcommand{\R}{\mathbb{R}}
\newcommand{\N}{\mathbb{N}}

\newcommand{\pr}[1]{\mathrm{proj}^{ss}_{#1}}
\newcommand{\proj}{\mathrm{proj}}

\newcommand{\xv}{\underline{x}}
\newcommand{\yv}{\underline{y}}
\newcommand{\ii}{\mathbf{i}}

\newcommand{\jj}{\mathbf{j}}
\newcommand{\tv}{\underline{t}}
\newcommand{\ldim}{\dim_{\mathrm{Lyap}}}
\newcommand{\pv}{\underline{p}}
\newcommand{\tvv}{\mathbf{t}}
\newcommand{\clo}{\overline{\mathcal{O}}}

\newcommand{\Alpha}{\mathcal{A}}
\newcommand{\cm}[2]{\widehat{\mu}^{#2}_{#1}}
\newcommand{\y}{\underline{\mathbf{y}}}
\newcommand{\z}{\underline{\mathbf{z}}}
\newcommand{\sgn}{\mathrm{sgn}}
\newcommand{\diam}{\mathrm{diam}}
\newcommand{\spt}{\mathrm{spt}}

\DeclareMathOperator*{\esssup}{ess\,sup}

\begin{document}
\title[On the Ledrappier-Young formula for self-affine measures] {On the Ledrappier-Young formula for self-affine measures}

\author{Bal\'azs B\'ar\'any}
\address[Bal\'azs B\'ar\'any]{Budapest University of Technology and Economics, MTA-BME Stochastics Research Group, P.O.Box 91, 1521 Budapest, Hungary \\ Mathematics Institute, University of Warwick, Coventry CV4 7AL, UK} \email{balubsheep@gmail.com}

\thanks{The research of B\'ar\'any was supported by the grants EP/J013560/1 and OTKA K104745.}

\begin{abstract}
Ledrappier and Young introduced a relation between entropy, Lyapunov exponents and dimension for invariant measures of diffeomorphisms on compact manifolds. In this paper, we show that a self-affine measure on the plane satisfies the Ledrappier-Young formula if the corresponding iterated function system (IFS) satisfies the strong separation condition and the linear parts satisfy the dominated splitting condition. We give sufficient conditions, inspired by Ledrappier and by Falconer and Kempton, that the dimensions of such a self-affine measure is equal to the Lyapunov dimension. We show some applications, namely, we give another proof for Hueter-Lalley's theorem and we consider self-affine measures and sets generated by lower triangular matrices.
\end{abstract}
\date{\today}

\maketitle

\thispagestyle{empty}

\section{Introduction}

Let $\Alpha:=\left\{A_1, A_2,\dots, A_N\right\}$ be a finite set of contracting, non-singular $2\times2$ matrices, and let $\Phi:=\left\{f_i(\xv)=A_i\xv+\tv_i\right\}_{i=1}^N$ be an {\it iterated function system} on the plane with affine mappings, where $\underline{t}_i\in\R^2$ for $i=1,\dots,N$. It is a well-known fact that there exists an unique non-empty compact subset $\Lambda$ of $\R^2$ such that $$\Lambda=\bigcup_{i=1}^Nf_i(\Lambda).$$ We call the set $\Lambda$ the {\it attractor} of $\Phi$.

Throughout the paper we denote the Hausdorff dimension of a set $A$ by $\dim_HA$, the packing dimension by $\dim_PA$, the lower and upper box counting dimension by $\underline{\dim}_BA$ and $\overline{\dim}_BA$ and the box counting dimension by $\dim_BA$. For the definitions and basic properties, we refer to Falconer~\cite{Fb1}.

The dimension theory of self-affine sets is far away from being well understood. One of the most natural approaches for the Hausdorff and box dimension of self-affine sets is the {\em subadditive pressure function}, introduced by Falconer \cite{F}. Denote by $\alpha_i(A)$ the $i$th {\it singular value} of a $2\times2$ non-singular matrix $A$, i.e. the positive square root of the $i$th eigenvalue of $AA^*$, where $A^*$ is the transpose of $A$. For $s\geq0$ define the {\it singular value function} $\phi^s$ as follows
\begin{equation*}
\phi^s(A):=\left\{\begin{array}{cc}
               \alpha_1(A)^s & 0\leq s\leq1 \\
               \alpha_1(A)\alpha_2(A)^{s-1} & 1<s\leq2 \\
               \left(\alpha_1(A)\alpha_2(A)\right)^{s/2} & s>2.
             \end{array}\right.
\end{equation*} We note that in this case, $\alpha_1(A)=\|A\|$ and $\alpha_2(A)=\|A^{-1}\|^{-1}$, where $\|.\|$ is the usual matrix norm induced by the Euclidean norm on $\R^2$. Let us define the subadditive pressure function generated by $\Alpha$ for $s\geq0$ as
\begin{equation}\label{esap}
P(s):=\lim_{n\rightarrow\infty}\frac{1}{n}\log\sum_{i_1,\dots,i_n=1}^N\phi^s(A_{i_1}\cdots A_{i_n}).
\end{equation} The function $P(s)$ is continuous, strictly monotone decreasing on $[0,\infty)$, moreover $P(0)=\log N$ and $\lim_{s\rightarrow\infty}P(s)=-\infty$. Falconer showed in \cite{F} that the unique root $s_0$ of the subadditive pressure function is always an upper bound for the box dimension of the attractor $\Lambda$ and if $\|A_i\|<1/3$ for every $i=1,\dots,N$ then $$\dim_H\Lambda=\dim_B\Lambda=\min\left\{2,s_0\right\}\text{ for Lebesgue-almost every $\tvv=(\tv_1,\dots,\tv_N)\in\R^{2N}$.}$$ The bound was later extended to $1/2$ by Solomyak, see \cite{S}.

In the case of similarities (i.e. $A_i=\rho_iU_i$, where $0<\rho_i<1$ and $U_i$ are orthonormal matrices) the dimension theory of the attractors is well understood if a separation condition holds. In the case of strict affine mappings, it is very unclear. Bedford \cite{Be} and McMullen \cite{Mc} introduced independently a family of self-affine sets on the plane, where the Hausdorff and box dimension differs, however a separation condition holds. Later, such examples were constructed by Gatzouras and Lalley \cite{GL} and Bara\'nski \cite{B}. In these cases the linear parts of the maps were diagonal matrices.

Falconer \cite{F2} proved that under some conditions and separation, the box dimension of a self-affine set is equal to the root of the subaddtive pressure. However, the only known sufficient condition in general was given by Hueter and Lalley \cite{HL}, which ensures that the Hausdorff and box dimension of a self-affine set coincide and equal to the root of the subadditive pressure. Recently, Falconer and Kempton~\cite{FK} gave conditions which ensure similar consequences.

One way to understand the Hausdorff dimension of self-affine sets depends on understanding of Hausdorff dimension of self-affine measures. We call a measure $\mu$ {\it self-affine} if it is compactly supported with support $\Lambda$ and there exists a $\pv=(p_1,\dots,p_N)$ probability vector such that
\begin{equation}\label{edefsameasure}
\mu=\sum_{i=1}^Np_i\mu\circ f_i^{-1}.
\end{equation}Ledrappier and Young \cite{LY1,LY2} introduced a formula for the Hausdorff dimension of invariant measures of diffeomorphisms on compact manifolds. It is a widespread claim that self-affine measures satisfy this formula but it was proven just in a very few cases. Basically, the first result on a class of self-affine measures and sets, for which the formula hold, was proven by Przytycki and Urba\'nski \cite{PU}. Later, Feng and Hu \cite{FH} proved that if the linear parts of the mappings are diagonal matrices then the Ledrappier-Young formula holds for the Hausdorff dimension of the self-affine measures, without assuming any separation condition or condition on the norm of the matrices. Moreover, Ledrappier \cite{L} proved that the formula is valid for a special family of self-affine measures, namely when the support is the graph of a Weierstrass functions.

Our main goal is to generalize Ledrappier's result \cite{L} for a more general family of self-affine measures.

Another important dimension theoretical property of a self-affine measure is its exactness. Denote by $B_r(\xv)$ the two dimensional ball centered at $\xv\in\R^2$ with radius $r$. Then we call
\begin{equation*}
\underline{d}_{\mu}(\xv)=\liminf_{r\rightarrow0+}\frac{\log\mu(B_r(\xv))}{\log r}\text{ and }\overline{d}_{\mu}(\xv)=\limsup_{r\rightarrow0+}\frac{\log\mu(B_r(\xv))}{\log r}
\end{equation*}
the {\it lower and upper local dimension} of $\mu$ at the point $\xv$, if the limit exists then we say that the measure has {\it local dimension} $d_{\mu}(\xv)$ at the point $\xv$. It is well-known fact that \begin{equation}\label{elocdimHdim}
\dim_H\mu=\esssup_{\xv\in\spt\mu}\underline{d}_{\mu}(\xv)=\inf\left\{\dim_HA:\mu(A^c)=0\right\}
\end{equation} for any $\mu$ Radon measure, where $\spt\mu$ denotes the support of $\mu$ and $A^c$ denotes the complement of $A$, see \cite{FLR}. Moreover, we call the measure $\mu$ {\it exact dimensional} if the local dimension exists at $\mu$-almost every points and equals $\dim_H\mu$. Feng and Hu \cite{FH} proved that self-similar measures, and self-affine measures if the linear parts are diagonal matrices, are exact dimensional. Ledrappier \cite{L} proved this for the graphs of Weierstrass functions, a phenomena that we also extend.

To analyse self-affine measures, it is convenient to handle it as a natural projection of Bernoulli measures. That is, let $\Sigma^+=\left\{1,\dots,N\right\}^{\N}$ be the {\it symbolic space} of one-sided infinite length words and let $\nu=\left\{p_1,\dots,p_N\right\}^{\N}$ be a {\it Bernoulli measure}, where $\pv=(p_1,\dots,p_N)$ is a probability vector. If $\pi_+:\Sigma^+\mapsto\Lambda$ denotes the {\it natural projection}, i.e. $\pi_+(i_0,i_1,\dots)=\lim_{n\rightarrow\infty}f_{i_0}\circ \cdots\circ f_{i_n}(\underline{0})$, then $\mu=(\pi_+)_*\nu=\nu\circ\pi_+^{-1}$.

According to the result of Oseledec Multiplicative Ergodic Theorem \cite{O} for $\nu$-almost every $\ii\in\Sigma^+$ there exist constants $0<\ly{s}\leq\ly{ss}$ such that
\begin{eqnarray*}
&& \lim_{n\rightarrow\infty}\frac{1}{n}\log\alpha_1(A_{i_0}\cdots A_{i_{n-1}})=-\ly{s}\text{ and }\\
&& \lim_{n\rightarrow\infty}\frac{1}{n}\log\alpha_2(A_{i_0}\cdots A_{i_{n-1}})=-\ly{ss}\text{ for $\nu$-a.e. $\ii=(i_0,i_1,\dots)\in\Sigma^+$.}
\end{eqnarray*} We call the constants $\ly{s}$ and $\ly{ss}$ the {\it Lyapunov exponents}. Denote the entropy of $\nu$ by $h_{\nu}=-\sum_{i=1}^Np_i\log p_i$; then we  define the {\it Lyapunov-dimension} of the measure $\mu$ by
\begin{equation*}
\ldim\mu=\min\left\{2,\frac{h_{\nu}}{\ly{s}},1+\frac{h_{\nu}-\ly{s}}{\ly{ss}}\right\}.
\end{equation*} Jordan, Pollicott and Simon showed that the Lyapunov dimension of a self-affine measure is always an upper bound for the Hausdorff dimension, see \cite{JPS}. We show also a sufficient condition (based on the idea of Ledrappier \cite{L}) which implies that the Lyapunov and Hausdorff dimension of a self-affine measure coincide. 

Throughout the paper we will follow the method of Ledrappier \cite{L} and Ledrappier and Young \cite{LY1, LY2}. At the end of the paper we give an alternative proof for the Hueter-Lalley Theorem and we show some applications for triangular matrices.

\section{Preliminaries and Results}

Let $\Alpha:=\left\{A_1, A_2,\dots, A_N\right\}$ be a finite set of contracting, non-singular $2\times2$ matrices, and let $\Phi:=\left\{f_i(\xv)=A_i\xv+\tv_i\right\}_{i=1}^N$ be an iterated function system on the plane with affine mappings.

\begin{definition}\label{dSSC}
We say that $\Phi$ satisfies the \textnormal{strong separation condition} (SSC) if there exists an open, non-empty and bounded set $\mathcal{O}\subset\R^2$ such that
\begin{enumerate}
  \item for every $i=1,\dots,N$, $f_i(\clo)\subseteq\mathcal{O}$ and
  \item for every $i\neq j$, $f_i(\clo)\cap f_j(\clo)=\emptyset$,
\end{enumerate}
where $\overline{\mathcal{O}}$ denotes the closure of $\mathcal{O}$.
\end{definition}

If the IFS satisfies the SSC then
\begin{equation}\label{eSSC}
f_i(\Lambda)\cap f_j(\Lambda)=\emptyset\text{ for every $i\neq j$,}
\end{equation}where $\Lambda$ denotes the attractor of $\Phi$. One can show that \eqref{eSSC} is actually equivalent to SSC. Moreover, $$\Lambda=\bigcap_{n=1}^{\infty}\bigcup_{i_1,\dots,i_n=1}^Nf_{i_1}\circ\cdots\circ f_{i_n}(\overline{\mathcal{O}}).$$ 

Let us denote by $\mathcal{S}=\left\{1,\dots,N\right\}$ the set of symbols and by $\Sigma=\mathcal{S}^{\mathbb{Z}}$ the symbolic space of two-sided infinite words. Moreover, let $\Sigma^{+}=\mathcal{S}^{\mathbb{N}}$ be the set of right- and $\Sigma^{-}=\mathcal{S}^{\mathbb{Z}_-}$ be the set of left side infinite length words. We note that in our definition of natural numbers, $0\in\N$. For a two-sided infinite length word $\ii=(\dots,i_{-2},i_{-1};i_0,i_1,i_2,\dots)$ let us denote the left hand side by $\ii_-$ and the right-hand side by $\ii_+$, i.e. $\ii_-=(\dots,i_{-2},i_{-1})$ and $\ii_+=(i_0,i_1,i_2,\dots)$. Denote by $\Sigma^{*}=\bigcup_{n=0}^{\infty}\mathcal{S}^n$ the set of finite length words. The number of symbols in a finite length word $\underline{i}$ is denoted by $|\underline{i}|$ and for an infinite word $\ii\in\Sigma$ we denote by $\ii|_n^k$ the elements of $\ii$ between $n$ and $k$, i.e. $\ii|_n^k=(i_n,\dots,i_k)$. Let us define also the cylinder sets on $\Sigma$ (and on $\Sigma^{+}$ respectively) by $$[\ii|_n^k]=\left\{\jj\in\Sigma:\jj|_n^k=\ii|^k_n\right\}.$$ We note that we consider $\Sigma^+$ with the usual topology, i.e. the topology generated by cylinder sets. This topology is metrizable with metric $d(\ii,\jj)=\beta^{\min\left\{k\geq0:\ii|_0^k\neq\jj|_0^k\right\}}$, where $0<\beta<1$.

We denote the composition of functions of $\Phi$ for a finite length word $\underline{i}=(i_1,\dots,i_n)\in\Sigma^*$ by $f_{\underline{i}}=f_{i_1}\circ\cdots\circ f_{i_n}$.

Now let us introduce a dynamical system $F$ acting on $\clo\times\Sigma^{+}$ by
\begin{equation*} 
F(\xv,\ii):=(f_{i_0}(\xv),\sigma\ii),
\end{equation*}
where $\mathcal{O}$ is the open and bounded set from Definition~\ref{dSSC}. Since $F$ is a hyperbolic map acting $\clo\times\Sigma^{+}$, the unique non-empty and compact set, which is $F$-invariant, is $\bigcap_{n=0}^{\infty}F^n(\clo\times\Sigma^+)=\Lambda\times\Sigma^+$.

Define $\pi_-:\Sigma^-\mapsto\Lambda$ (similarly to $\pi_+$) by
\begin{equation*}
\pi_-(\dots,i_{-2},i_{-1})=\lim_{n\rightarrow\infty}f_{i_{-1}}\circ \cdots\circ f_{i_{-n}}(0)=\sum_{n=1}^{\infty}A_{i_{-1}}\cdots A_{i_{-n+1}}\tv_{i_{-n}}.
\end{equation*}
If $\sigma$ is the left-shift operator on $\Sigma$ then it is easy to see that $F$ is conjugate to $\sigma$ by the projection $\pi:\Sigma\mapsto\Lambda\times\Sigma^{+}$, where $\pi(\ii):=(\pi_-(\ii_-),\ii_+)$. That is,
$$\pi\circ\sigma=F\circ\pi.$$ 

Let $\pv=(p_1,\dots,p_N)$ be a probability vector and let $\nu=\left\{p_1,\dots,p_N\right\}^{\N}$ be the corresponding left-shift invariant and ergodic Bernoulli-probability measure on $\Sigma^{+}$. Denote by $\widehat{\nu}=\left\{p_1,\dots,p_N\right\}^{\mathbb{Z}}$ the natural extension of $\nu$ to $\Sigma$. Let us define its projection to $\Lambda\times\Sigma^{+}$ by $\widehat{\mu}:=\pi_*\widehat{\nu}=\widehat{\nu}\circ\pi^{-1}$. Then $\widehat{\mu}$ is a $F$-invariant and ergodic probability measure on $\Lambda\times\Sigma^+$, moreover $\widehat{\mu}=\mu\times\nu$, where $\mu$ is self-affine measure defined in \eqref{edefsameasure}.

For the analysis of the dimension theoretical point of view, we need an assumption for the matrices $\Alpha$, which ensures for us that there is a dynamically invariant foliation on $\clo\times\Sigma^+$.

\begin{definition}\label{ddomsplit}
We say that the set $\Alpha$ of matrices satisfies the \textnormal{dominated splitting condition} if there are constants $C,\delta>0$ such that
\[
\frac{\alpha_1(A_{\underline{i}})}{\alpha_2(A_{\underline{i}})}\geq Ce^{\delta n}\text{ for all $\underline{i}\in\Sigma^*$ with $|\underline{i}|=n$.}
\]
\end{definition}

For example, a family of matrices with strictly positive entries satisfies dominated splitting, see \cite{ABY}.

Let us define a map from $\Sigma$ to $\Alpha$ in a natural way, i.e. $A(\ii):=A_{i_0}$. Denote the product by $A^{(n)}(\ii):=A(\sigma^{n-1}\ii)\cdots A(\ii)$ for $\ii\in\Sigma$ and $n\geq1$. Now we are going to state some useful properties for set $\Alpha$ of matrices, satisfying dominated splitting.

\begin{lemma}[\cite{BG},\cite{Y}]\label{ldomsplit1}
The set $\Alpha$ of matrices satisfies the dominated splitting condition if and only if for every $\ii\in\Sigma$ there are two one-dimensional subspaces $e_{ss}(\ii),e_s(\ii)$ of $\R^2$ such that
\begin{enumerate}
  \item $A(\ii)e_i(\ii)=e_i(\sigma\ii)$ for every $\ii\in\Sigma$ and $i=s,ss$,
  \item there are constants $C,\delta>0$ such that $$\frac{\|A^{(n)}(\ii)|e_s(\ii)\|}{\|A^{(n)}(\ii)|e_{ss}(\ii)\|}\geq Ce^{\delta n}\text{ for all $\ii\in\Sigma$ and $n\geq1$.}$$
\end{enumerate}
We call the family of subspaces $e_{ss}$ \textnormal{strong stable directions}.
\end{lemma}

 We note that the dependence of the subspaces $e_i$ on $\ii\in\Sigma$ is continuous, that is $e_i:\Sigma\mapsto\mathbf{P}^1$ is continuous with the standard metrics, where $\mathbf{P}^1$ denotes the projective space, see \cite[Section~B.1]{BDV}.

\begin{lemma}[\cite{BR}]\label{ldomsplit2}
Let $\Alpha$ be a set of matrices satisfying the dominated splitting condition and let $e_{ss}(\ii),e_s(\ii)$ be the two one-dimensional subspaces of $\R^2$ defined in Lemma~\ref{ldomsplit1}. Then there exists a constant $C>0$ such that
\begin{eqnarray*}
&& C^{-1}\|A^{(n)}(\ii)|e_s(\ii)\|\leq\alpha_1(A^{(n)}(\ii))\leq C\|A^{(n)}(\ii)|e_s(\ii)\|\text{ and }\\
&& C^{-1}\|A^{(n)}(\ii)|e_{ss}(\ii)\|\leq\alpha_2(A^{(n)}(\ii))\leq C\|A^{(n)}(\ii)|e_{ss}(\ii)\|.
\end{eqnarray*}
In particular,
\begin{equation}\label{elyapint}
\ly{i}=-\lim_{n\rightarrow\infty}\frac{1}{n}\log\|A^{(n)}(\ii)|e_i(\ii)\|=-\int\log\|A(\ii)|e_i(\ii)\|d\widehat{\nu}(\ii)\text{ for $\widehat{\nu}$-a.e. $\ii$ and }i=s,ss.
\end{equation}
\end{lemma}

The dominated splitting property implies that the Lyapunov exponents are always separated. Actually, $\ly{s}+\delta\leq\ly{ss}$ for any self-affine measure $\mu$, where $\delta$ is in Definition~\ref{ddomsplit}.

Let $C_+:=\left\{(x,y)\in\R^2\backslash\{(0,0)\}:xy\geq0\right\}$ be the {\it standard positive cone}. A {\it cone} is an image of $C_+$ by a linear isomorphism and a {\it multicone} is a disjoint union of finitely many cones.

\begin{lemma}[\cite{ABY}, \cite{BG}]\label{ldomsplit3a}
A set $\Alpha$ of matrices satisfies dominated splitting condition if and only if $\Alpha$ has a \textnormal{forward invariant multicone}, i.e there is a multicone $M$ such that $\bigcup_{i=1}^NA_i(M)\subset M^o$, where $M^o$ denotes the interior of $M$.
\end{lemma}

Note that if $M$ is a forward-invariant multicone w.r.t $\Alpha=(A_1,\dots,A_N)$ then the closure of its complement is {\it backward-invariant multicone}, i.e. forward-invariant for $\Alpha^{-1}=(A_1^{-1},\dots,A_N^{-1})$.

\begin{lemma}[\cite{BR}]\label{ldomsplit3b}
Let $\Alpha$ be a set of matrices satisfying the dominated splitting condition and let $M$ be a forward-invariant multicone. Then for every $\ii\in\Sigma$
$$e_s(\ii)=\bigcap_{n=1}^{\infty}A_{i_{-1}}\cdots A_{i_{-n}}(M)\text{ and }e_{ss}(\ii)=\bigcap_{n=1}^{\infty}A_{i_{0}}^{-1}\cdots A_{i_{n-1}}^{-1}(\overline{M^c}),$$
where $M^c$ denotes the complement of $M$. In particular, $e_s(\ii)$ depends only on $\ii_-$ and $e_{ss}(\ii)$ depends only on $\ii_+$.
\end{lemma}

An easy consequence of Lemma~\ref{ldomsplit3a} and Lemma~\ref{ldomsplit3b} is that the included angle of $e_s(\ii),e_{ss}(\jj)$ is uniformly separated away from zero for every $\ii,\jj\in\Sigma$.

Let us denote the orthogonal projection from $\R^2$ to the subspace perpendicular to $\theta\in\mathbf{P}^1$ by $\proj_{\theta}$. For simplicity, we denote the orthogonal projection $\proj_{e_{ss}(\ii)}$ by $\pr{\ii}$. We call the family of projections of $\mu$ along the strong stable directions {\em transversal measures} and we denote by
\begin{equation}\label{etransmeasure}
\mu^T_{\ii}:=(\pr{\ii})_*\mu=\mu\circ(\pr{\ii})^{-1}.
\end{equation} Now we are ready to state our main theorem.

\begin{theorem}\label{tmain1}
Let $\Alpha=\left\{A_1, A_2,\dots, A_N\right\}$ be a finite set of contracting, non-singular $2\times2$ matrices, and let $\Phi=\left\{f_i(\xv)=A_i\xv+\tv_i\right\}_{i=1}^N$ be an iterated function system on the plane with affine mappings. Let $\nu$ be a left-shift invariant and ergodic Bernoulli-probability measure on $\Sigma^{+}$, and $\mu$ be the corresponding self-affine measure. If
\begin{enumerate}
    \item $\Alpha$ satisfies dominated splitting,
    \item $\Phi$ satisfies the strong separation condition
\end{enumerate} then $\mu$ is exact dimensional and
\begin{equation}\label{eLYformula}
\dim_H\mu=\frac{h_{\nu}}{\ly{ss}}+\left(1-\frac{\ly{s}}{\ly{ss}}\right)\dim_H\mu^{T}_{\ii}\text{ for $\nu$-almost every $\ii\in\Sigma^+$,}
\end{equation}
where $h_{\nu}$ denotes the entropy of $\nu$ and $\ly{s},\ly{ss}$ are the Lyapunov exponents, defined in \eqref{elyapint}.
\end{theorem}

We note that, \eqref{eLYformula} implies that $\dim_H\mu^T_{\ii}$ is constant for $\nu$-a.e. $\ii\in\Sigma^+$. In particular, $\mu^T_{\ii}$ is exact dimensional with constant dimension for $\nu$-a.e $\ii$, see Proposition~\ref{pLY2}.

It is a non-trivial question, how the strong separation condition can be relaxed to the open set condition (OSC). Let $A_1$ and $A_2$ be two matrices with strictly positive entries such that the IFS $\left\{f_i(\xv)=A_i\xv\right\}_{i=1,2}$ maps the closed unit square into itself and $f_1((0,1)^2)\cap f_2((0,1)^2)=\emptyset$. Then the IFS satisfies the open set condition, however its attractor is only a single point. Hence, \eqref{eLYformula} cannot hold for any self-affine measure, which are just the Dirac measure. However, we conjecture that if the attractor contains at least two points and the IFS satisfies the OSC then \eqref{eLYformula} holds. 

Since the transversal measures $\mu^{T}_{\ii}$ are the orthogonal projections of $\mu$, $\dim_H\mu^{T}_{\ii}\leq\min\left\{1,\dim_H\mu\right\}$. By \eqref{eLYformula}, simple algebraic manipulations show that
\begin{equation}\label{elyapHdimequal}
\dim_H\mu=\ldim\mu\Leftrightarrow\dim_H\mu^{T}_{\ii}=\min\left\{1,\dim_H\mu\right\}\text{ for $\nu$-a.e. $\ii\in\Sigma^+$.}
\end{equation} If the distribution of the strong stable directions $e_{ss}$ has large dimension then one can claim that the right-hand side of \eqref{elyapHdimequal} holds. Let us consider the map $e_{ss}:\Sigma^+\mapsto\mathbf{P}^1$ which maps an $\ii\in\Sigma^+$ to the element of the projective space associated to $e_{ss}(\ii)$. Let us define the push-down measure of $\nu$ by $e_{ss}$ on $\mathbf{P}^1$ as
\begin{equation}\label{edirections}
\nu_{ss}:=(e_{ss})_*\nu=\nu\circ(e_{ss})^{-1}.
\end{equation}

\begin{theorem}\label{tmain1b}
Let $\Alpha=\left\{A_1, A_2,\dots, A_N\right\}$ be a finite set of contracting, non-singular $2\times2$ matrices, and let $\Phi=\left\{f_i(\xv)=A_i\xv+\tv_i\right\}_{i=1}^N$ be an iterated function system on the plane with affine mappings. Let $\nu$ be a left-shift invariant and ergodic Bernoulli-probability measure on $\Sigma^{+}$, and $\mu$ be the corresponding self-affine measure. If
\begin{enumerate}
    \item $\Alpha$ satisfies dominated splitting,
    \item $\Phi$ satisfies the strong separation condition,
    \item $\dim_H\nu_{ss}\geq\min\left\{1,\ldim\mu\right\}$
\end{enumerate}then
\begin{equation*}
\dim_H\mu=\ldim\mu=\min\left\{\frac{h_{\nu}}{\ly{s}},1+\frac{h_{\nu}-\ly{s}}{\ly{ss}}\right\}.
\end{equation*}
\end{theorem}

The proof of Theorem~\ref{tmain1b} is based on the idea of Ledrappier \cite[Lemma~1]{L}. It uses an extension of the result of Marstrand \cite{Ma}, which was obtained by Kaufman~\cite{K}. Kaufman~\cite{K} showed that for any Borel subset $A$ of $\R^2$ the exceptional set of directions, where the Hausdorff dimension of orthogonal projection drops, has dimension at most $\min\left\{1,\dim_HA\right\}$. We use this phenomena for orthogonal projections of measures. Because of later usage we show a modified version in Lemma~\ref{lmarstrand}.

\begin{proof}[Proof of Theorem~\ref{tmain1b}]
By Theorem~\ref{tmain1}, we know that $\dim_H\mu^{T}_{\ii}$ is a constant for $\nu$-almost every $\ii\in\Sigma^+$. Using Lemma~\ref{lmarstrand} we have that for every $\varepsilon>0$ there exists a set $A\subseteq\Sigma^+$ such that $\nu(A)>0$ and for every $\ii\in A$ $\dim_H\mu^{T}_{\ii}\geq\min\left\{1,\dim_H\mu\right\}-\varepsilon.$ This implies that $\dim_H\mu^{T}_{\ii}\geq\min\left\{1,\dim_H\mu\right\}-\varepsilon$ for $\nu$-almost every $\ii\in\Sigma^+$. Since $\varepsilon>0$ was arbitrary we get $$\dim_H\mu^{T}_{\ii}=\min\left\{1,\dim_H\mu\right\}\text{ for $\nu$-almost every $\ii\in\Sigma^+$}.$$ The statement of the theorem follows by \eqref{elyapHdimequal}.
\end{proof}

Another upper estimate on the dimension of exceptional directions, where the dimension of orthogonal projection of Borel subsets $A$ of $\R^2$ drops, is $\min\left\{1,2-\dim_HA\right\}$. This result was showed by Falconer~\cite{Fex}. We can use this estimate for the orthogonal projections of self-affine measures to ensure that the Hausdorff and Lyapunov dimension coincide. We adapt here the recent result of Falconer and Kempton~\cite{FK} for self-affine measures.

\begin{theorem}\label{tmain1c}
	Let $\Alpha=\left\{A_1, A_2,\dots, A_N\right\}$ be a finite set of contracting, non-singular $2\times2$ matrices, and let $\Phi=\left\{f_i(\xv)=A_i\xv+\tv_i\right\}_{i=1}^N$ be an iterated function system on the plane with affine mappings. Let $\nu$ be a left-shift invariant and ergodic Bernoulli-probability measure on $\Sigma^{+}$, and $\mu$ be the corresponding self-affine measure. If
	\begin{enumerate}
		\item $\Alpha$ satisfies dominated splitting,
		\item $\Phi$ satisfies the strong separation condition,
		\item $\dim_H\nu_{ss}+\dim_H\mu>2$ \label{cmain1c}
	\end{enumerate}then
	\begin{equation}\label{elyapHdim2}
	\dim_H\mu=\ldim\mu=1+\frac{h_{\nu}-\ly{s}}{\ly{ss}}>1.
	\end{equation}		
\end{theorem}

\begin{proof}
	By Theorem~\ref{tmain1}, the measure $\mu$ is exact dimensional. Thus, by Egorov's Theorem for every $\varepsilon>0$ there exists a set $\Omega\subseteq\Lambda$ such that $\mu(\Omega)>1-\varepsilon$ and
	$$
	\int_{\Omega}\int_{\Omega}\frac{1}{\|\xv-\yv\|^{\dim_H\mu-\varepsilon}}d\mu(\xv)d\mu(\yv)<\infty.
	$$
	Let us fix $\varepsilon>0$ such that $\dim_H\nu_{ss}+\dim_H\mu>2+\varepsilon$
	By applying Peres and Schlag \cite[Proposition~6.1]{PS}, we get
	$$
	\dim_H\left\{\theta\in\mathbf{P}^1:\dim_H(\proj_{\theta})_*\mu<1\right\}\leq2-\dim_H\mu+\varepsilon.
	$$
	Since $\dim_H\nu_{ss}>2-\dim_H\mu+\varepsilon$ we have $\dim_H\mu^T_{\ii}=1$ for $\nu$-a.e. $\ii\in\Sigma^+$. Formula \eqref{elyapHdim2} follows by \eqref{eLYformula}.
\end{proof}

A discussion on possible applications for Theorem~\ref{tmain1c} is given in Theorem~\ref{tmain1capp}.

The statement of Theorem~\ref{tmain1} does not follow directly from the result of Ledrappier and Young \cite[Theorem~C', Corollary~D']{LY2}. The dynamical system~$F$, which is induced naturally by the IFS~$\Phi$, does not act on a Riemannian manifold without boundary. It can be conjugated to a dynamical system acting on a compact Riemannian manifold without boundary, but it would be piecewise smooth and would contain singularities, hence it wouldn't be a diffemorphism.  However, the properties of $F$, which are implied by dominated splitting, allow us to adapt the proofs and methods of \cite{L} and \cite{LY2}. 

The proof of Theorem~\ref{tmain1} is decomposed into four propositions, Proposition~\ref{pLY1}, \ref{pLY2}, \ref{pLY3}, and \ref{pLY4}. In Proposition~\ref{pLY1} we show the exact dimensionality of the components of the affine measure in the strong stable directions and also find their dimension, whilst in Proposition~\ref{pLY2} we do the same for the transversal measures. The proofs of Proposition~\ref{pLY1} and \ref{pLY2} follow the proof of \cite[Proposition~2]{L}. Then we show that the measure $\mu$ has a product structure in dimension, that is, the dimension of $\mu$ is the sum of the dimension of the strong stable components and the dimension of transversal measure. This fact is showed in two parts in Proposition~\ref{pLY3} and Proposition~\ref{pLY4}. The proof of Proposition~\ref{pLY3} is a modified version of \cite[Lemma~11.3.1]{LY2} and Proposition~\ref{pLY4} is a modification of \cite[Section~(10.2)]{LY2}.

\section{Proof of the Ledrappier-Young formula}

Let $\nu$ be the left-shift invariant and ergodic Bernoulli-probability measure on $\Sigma^{+}$ and $\mu$ be the self-affine measure defined in \eqref{edefsameasure}. Let $\widehat{\mu}=\mu\times\nu$ be the $F$-invariant and ergodic probability measure on $\Lambda\times\Sigma^+$, defined in the previous section. Denote by $\mathcal{B}$ the usual Borel $\sigma$-algebra on $\Lambda\times\Sigma^{+}$.

If $\zeta$ is a measurable partition of $\Lambda\times\Sigma^+$ then by the result of Rokhlin \cite{R}, there exists a canonical system of conditional measures, i.e. for $\widehat{\mu}$-a.e. $\y\in\Lambda\times\Sigma$ there exists a measure $\mu^{\zeta}_{\y}$ supported on $\zeta(\y)$, the element of $\zeta$ containing $\y$, such that for every measurable set $A$ the function $\y\mapsto\mu_{\y}^{\zeta}(A)$ is $\mathcal{B}_{\zeta}$-measurable, where $\mathcal{B}_{\zeta}$ is the sub-$\sigma$-algebra of $\mathcal{B}$ whose elements are union of elements of $\zeta$, and
\begin{equation}\label{econdmeasure}
\widehat{\mu}(A)=\int\widehat{\mu}_{\y}^{\zeta}(A)d\widehat{\mu}(\y).
\end{equation}
The conditional measures are uniquely defined up to a set of zero measure.

For two measurable partitions $\zeta_1$ and $\zeta_2$ we define the common refinement $\zeta_1\vee\zeta_2$ such that for every $\y$, $(\zeta_1\vee\zeta_2)(\y)=\zeta_1(\y)\cap\zeta_2(\y)$. Moreover, let us define the image of the partition $\zeta$ in the natural way, i.e. for every $\y$, $(F\zeta)(\y)=F(\zeta(F^{-1}(\y)))$. 

Now, we define a dynamically invariant foliation on $\Lambda\times\Sigma^+$ with respect to the strong stable directions. Denote by $e_{ss}$ the family of one-dimensional strong stable directions defined in Lemma~\ref{ldomsplit1}. Since $e_{ss}$ depends only on $\ii_+$ by Lemma~\ref{ldomsplit3a}, it defines a foliation on $\clo$ for every $\ii_+\in\Sigma^+$. Hence, it defines a foliation $\xi^{ss}$ on $\Lambda\times\Sigma^+$. Namely, for a $\y=(\xv,\ii)\in\Lambda\times\Sigma^+$ let $l_{ss}(\y)$ be the line through $\xv$ parallel to $e_{ss}(\ii)$ on $\R^2\times\left\{\ii\right\}$. Let the partition element $\xi^{ss}(\y)$ be the intersection of the line $l_{ss}(\y)$ with $\Lambda\times\left\{\ii\right\}$. It is easy to see that $F\xi^{ss}$ is a refinement of $\xi^{ss}$, that is, for every $\y$, $(F\xi^{ss})(\y)\subset\xi^{ss}(\y)$.

Let us define the conditional entropy of $F\xi^{ss}$ with respect to $\xi^{ss}$ in the usual way,
\begin{equation*} 
H(F\xi^{ss}|\xi^{ss}):=-\int\log\cm{\y}{\xi^{ss}}((F\xi^{ss})(\y))d\widehat{\mu}(\y).
\end{equation*}

Observe that if $\mathcal{Q}$ is a countable and measurable partition of $\Lambda\times\Sigma^+$ then 
\begin{equation}\label{econdcond}
	\left(\mu_{\y}^{\xi^{ss}}\right)_{\y}^{\mathcal{Q}}=\frac{\mu_{\y}^{\xi^{ss}}|_{\mathcal{Q}(\y)}}{\mu_{\y}^{\xi^{ss}}(\mathcal{Q}(\y))}=\mu_{\y}^{\xi^{ss}\vee\mathcal{Q}}\text{ for $\mu$-a.e. $\y$.}
\end{equation}
Indeed, 
\begin{multline*}
\int\mu_{\y}^{\xi^{ss}}d\mu(\y)=\int\sum_{Q\in\mathcal{Q}}\frac{\mu_{\y}^{\xi^{ss}}|_{Q}}{\mu_{\y}^{\xi^{ss}}(Q)}\mu_{\y}^{\xi^{ss}}(Q)d\mu(\y)=\iint\left(\mu_{\y}^{\xi^{ss}}\right)_{\z}^{\mathcal{Q}}d\mu_{\y}^{\xi^{ss}}(\z)d\mu(\y)=\\
\iint\left(\mu_{\z}^{\xi^{ss}}\right)_{\z}^{\mathcal{Q}}d\mu_{\y}^{\xi^{ss}}(\z)d\mu(\y)=\int\left(\mu_{\z}^{\xi^{ss}}\right)_{\z}^{\mathcal{Q}}d\mu(\z),
\end{multline*}
where we used that for $\widehat{\mu}$-a.e. $\y$, if $\z\in\xi^{ss}(\y)$ then $\mu^{\xi^{ss}}_{\y}=\mu^{\xi^{ss}}_{\z}$. Since for $\mu$-a.e. $\y$ the measure $\left(\mu_{\y}^{\xi^{ss}}\right)_{\y}^{\mathcal{Q}}$ is supported on $(\xi^{ss}\vee\mathcal{Q})(\y)$, by uniqueness of conditional measures we get \eqref{econdcond}.

\begin{prop}\label{pLY1}
For $\widehat{\mu}$-a.e. $\y\in\Lambda\times\Sigma^+$ the measure $\widehat{\mu}_{\y}^{\xi^{ss}}$ is exact dimensional and $$\dim_H\cm{\y}{\xi^{ss}}=\frac{H(F\xi^{ss}|\xi^{ss})}{\ly{ss}}.$$
\end{prop}

Before we prove the proposition, we define another partition $\mathcal{P}=\left\{f_i(\Lambda)\times\Sigma^{+}\right\}_{i=1}^N$. It is easy to see that
\begin{equation}\label{erefin}
\mathcal{P}\vee\xi^{ss}=F\xi^{ss}.
\end{equation}
Let us denote the ball with radius $r$ centered at $\y$ by $B_r(\y)$. Let $B^{ss}_r(\y)$ be the restriction of the ball to $\xi^{ss}(\y)$. That is, $$B^{ss}_r(\y)=\left\{\underline{\mathbf{z}}\in\xi^{ss}(\y):|\y-\underline{\mathbf{z}}|\leq r\right\},$$
where $|.|$ denotes the usual Euclidean norm on $\R^2$.

\begin{lemma}\label{ltoLY1}
There is a constant $c_1>0$ that for every $n\geq1$ and $\y=(\xv,\ii)\in\Lambda\times\Sigma^{+}$ with $\xv=\pi_{-}(\dots,i_{-2},i_{-1})$
\[
\left(\Lambda\times\left\{\ii\right\}\right)\cap B^{ss}_{c_1^{-1}\alpha_2(A_{i_{-1}}\cdots A_{i_{-n}})}(\y)\subseteq\left(\bigvee_{k=0}^{n-1}F^k\mathcal{P}\vee\xi^{ss}\right)(\y)\subseteq B^{ss}_{c_1\alpha_2(A_{i_{-1}}\cdots A_{i_{-n}})}(\y),
\]
where $\alpha_2(.)$ is the second singular value of a matrix.
\end{lemma}

\begin{proof}
Let us fix a $n\geq1$ and $\y\in\Lambda\times\Sigma^{+}$ and let $F^{-n}(\y)=(\xv',\ii')$ then $\ii'=(i_{-n},\dots,i_{-1},i_{0},\dots)$. Denote by $D=\mathrm{diam}(\clo)$ the diameter of $\clo$. By the definition of strong stable directions, see Lemma~\ref{ldomsplit1}, we have
$$\mathrm{diam}\left(\left(\bigvee_{k=0}^{n-1}F^k\mathcal{P}\vee\xi^{ss}\right)(\y)\right)\leq D\|A^{(n)}(\ii')|e_{ss}(\ii')\|.$$

On the other hand, let $\kappa=\min_{i\neq j}\mathrm{dist}(f_i(\Lambda),f_j(\Lambda))$. Since the IFS $\Phi$ satisfies the strong separation condition, see Definition~\ref{dSSC}, $\kappa>0$. Then for every $F^{-n}(\y)=(\xv',\ii')\in\Lambda\times\Sigma^+$ if $\xv'\in f_i(\Lambda)$ then $\mathrm{dist}(\xv',f_j(\Lambda))>\kappa/2$ for every $j\neq i$. So
\[
\Lambda\times\left\{\ii\right\}\cap F^{n}(B^{ss}_{\frac{\kappa}{2}}(F^{-n}(\y)))\subseteq\left(\bigvee_{k=0}^{n-1}F^k\mathcal{P}\vee\xi^{ss}\right)(\y).
\]
Applying again Lemma~\ref{ldomsplit1}, we get $F^{n}(B^{ss}_{\frac{\kappa}{2}}(F^{-n}(\y)))=B^{ss}_{\frac{\kappa}{2}\|A^{(n)}(\ii')|e_{ss}(\ii')\|}(\y)$.

Let $C>0$ be the constant defined in Lemma~\ref{ldomsplit2}, then by choosing $c_{1}:=C\max\left\{D,(\frac{\kappa}{2})^{-1}\right\}$, the statement of the lemma follows.
\end{proof}

\begin{proof}[Proof of Proposition~\ref{pLY1}]
To prove the statement of the proposition it is enough to show that
\begin{equation*}
\lim_{r\rightarrow0+}\frac{\log\cm{\y}{\xi^{ss}}(B^{ss}_r(\y))}{\log r}=\frac{H(F\xi^{ss}|\xi^{ss})}{\ly{ss}}\text{ for $\widehat{\mu}$-a.e $\y$.}
\end{equation*}

By Lemma~\ref{ltoLY1}, it is equivalent to show that
\begin{equation}\label{eentlyapLY1}
\lim_{n\rightarrow\infty}\frac{\log\cm{\y}{\xi^{ss}}\left(\left(\bigvee_{k=0}^{n-1}F^k\mathcal{P}\vee\xi^{ss}\right)(\y)\right)}{\log \alpha_2(A_{i_{-1}}\cdots A_{i_{-n}})}=\frac{H(F\xi^{ss}|\xi^{ss})}{\ly{ss}}\text{ for $\widehat{\mu}$-a.e $\y$.}
\end{equation}

We have
\begin{multline*}
\log\cm{\y}{\xi^{ss}}\left(\left(\bigvee_{k=0}^{n-1}F^k\mathcal{P}\vee\xi^{ss}\right)(\y)\right)=\log\cm{\y}{\xi^{ss}}\left(\mathcal{P}(\y)\cap\cdots\cap F^{n-1}(\mathcal{P}(F^{-n+1}(\y)))\right)=\\
\log\cm{\y}{\xi^{ss}}(\mathcal{P}(\y))\prod_{k=1}^{n-1}\frac{\cm{\y}{\xi^{ss}}\left(\mathcal{P}(\y)\cap\cdots\cap F^{k}(\mathcal{P}(F^{-k}(\y)))\right)}{\cm{\y}{\xi^{ss}}\left(\mathcal{P}(\y)\cap\cdots\cap F^{k-1}(\mathcal{P}(F^{-k+1}(\y)))\right)}.
\end{multline*}
By using \eqref{econdcond}, we get 
$$
\frac{\cm{\y}{\xi^{ss}}\left(\mathcal{P}(\y)\cap\cdots\cap F^{k}(\mathcal{P}(F^{-k}(\y)))\right)}{\cm{\y}{\xi^{ss}}\left(\mathcal{P}(\y)\cap\cdots\cap F^{k-1}(\mathcal{P}(F^{-k+1}(\y)))\right)}=\cm{\y}{\xi^{ss}\vee\mathcal{P}\vee\cdots\vee F^{k-1}\mathcal{P}}\left(F^{k}(\mathcal{P}(F^{-k}(\y)))\right).
$$
On the other hand, by applying \eqref{erefin}, $\xi^{ss}\vee\mathcal{P}\vee\cdots\vee F^{k-1}\mathcal{P}=F^{k}\xi^{ss}$. Moreover, by the invariance of the measure $\widehat{\mu}$
$$
\cm{\y}{F^{k}\xi^{ss}}\left(F^{k}(\mathcal{P}(F^{-k}(\y)))\right)=\cm{F^{-k}(\y)}{\xi^{ss}}\left(\mathcal{P}(F^{-k}(\y)))\right),
$$
Hence
\[
\frac{1}{n}\log\cm{\y}{\xi^{ss}}\left(\left(\bigvee_{k=0}^{n-1}F^k\mathcal{P}\vee\xi^{ss}\right)(\y)\right)=\frac{1}{n}\sum_{k=0}^{n-1}\log\cm{F^{-k}(\y)}{\xi^{ss}}\left(\mathcal{P}(F^{-k}(\y)))\right)
\]
Since $\widehat{\mu}$ is ergodic
\begin{equation}\label{eentLY1}
\lim_{n\rightarrow\infty}\frac{1}{n}\log\cm{\y}{\xi^{ss}}\left(\left(\bigvee_{k=0}^{n-1}F^k\mathcal{P}\vee\xi^{ss}\right)(\y)\right)=\int\log\cm{\y}{\xi^{ss}}(\mathcal{P}(\y))d\widehat{\mu}(\y)=-H(\mathcal{P}|\xi^{ss}).
\end{equation}
Using the property of conditional entropy and \eqref{erefin}, \begin{equation}\label{econdenteq}
H(\mathcal{P}|\xi^{ss})=H(\mathcal{P}\vee\xi^{ss}|\xi^{ss})=H(F\xi^{ss}|\xi^{ss}).
\end{equation} Applying Oseledec's Theorem, we have
$$\lim_{n\to\infty}\frac{1}{n}\log\alpha_2(A_{i_{-1}}\cdots A_{i_{-n}})=-\ly{ss}\text{ for $\nu$-a.e $\ii$,}$$
which together with \eqref{eentLY1} and \eqref{econdenteq} implies \eqref{eentlyapLY1}.
\end{proof}

The next proposition is devoted to proving that the transversal measures $\mu^T_{\ii}=\mu\circ(\pr{\ii})^{-1}$ are exact dimensional measures for $\nu$-a.e $\ii\in\Sigma^+$, and to calculating the typical Hausdorff dimension, where $\pr{\ii}$ is the orthogonal projection from $\R^2$ to the subspace perpendicular to $e_{ss}(\ii)$.

\begin{prop}\label{pLY2}
For $\nu$-a.e. $\ii\in\Sigma^+$ the measure $\mu^T_{\ii}$ is exact dimensional and $$\dim_H\mu^T_{\ii}=\frac{h_{\nu}-H(F\xi^{ss}|\xi^{ss})}{\ly{s}}.$$
\end{prop}

We define another invariant foliation $\xi^s$ with respect to the stable plane. That is, for every $\y=(\xv,\ii)\in\Lambda\times\Sigma^+$, $\xi^s(\y)=\Lambda\times\left\{\ii\right\}$. Then the foliation $\xi^s$ has similar properties to $\xi^{ss}$, i.e.  $F\xi^s$ is a refinement of $\xi^s$ and $\mathcal{P}\vee\xi^s=F\xi^s$. Moreover, it is easy to see that for every $\y$
\begin{equation}\label{establecondmeasure}
\cm{\y}{\xi^s}=\mu
\end{equation}

For the examination of the local dimension of the projected measure, instead of looking at the balls on the projection we introduce the transversal stable balls associated to the projection. Let $B^t_r(\xv,\ii)$ be transversal stable ball with radius $r$, i.e $$B^t_r(\xv,\ii)=\left\{(\yv,\jj):\ii=\jj\text{ and }\mathrm{dist}(l_{ss}(\xv,\ii),l_{ss}(\yv,\jj))\leq2r\right\},$$
where $l_{ss}(\xv,\ii)$ denotes the line through $\xv$ parallel to $e_{ss}(\ii)$.

For technical reasons we have to introduce the modified transversal stable ball. Since the IFS $\Phi$ satisfies the SSC, by Lemma~\ref{ldomsplit3b}, for a $\y=(\xv,\ii)\in\Lambda\times\Sigma^+$ we can define the stable direction $e_s(\y)$ of $\y$ by $e_s(\y):=e_s(\xv):=e_s(\ii_-)$, where $\pi_-(\ii_-)=\xv$. 

Then for a $(\xv,\ii)\in\Lambda\times\Sigma^+$, we define the modified transversal stable ball with radius $\delta$ by
\[
B^T_{\delta}(\xv,\ii)=\left\{(\yv,\jj)\in\Lambda\times\Sigma^+:\ii=\jj\text{ and }\mathrm{dist}_{e_s(\xv,\ii)}(l_{ss}(\xv,\ii),l_{ss}(\yv,\jj))\leq\delta\right\},
\]
where $\mathrm{dist}_{e_s(\xv,\ii)}(l_{ss}(\xv,\ii),l_{ss}(\yv,\jj))$ denotes the distance of the intersections of lines $l_{ss}(\xv,\ii)$ and $l_{ss}(\yv,\jj)$ with the subspace $e_s(\xv,\ii)$, see Figure~\ref{fdist}. 

\begin{figure}[h]
  \centering
  \includegraphics[width=70mm]{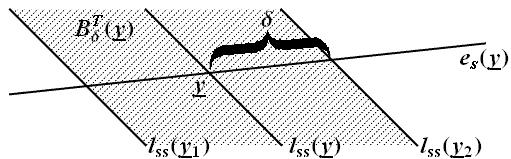}\\
  \caption{The modified transversal ball $B^T_{\delta}(\y)$.}\label{fdist}
\end{figure}

Since the included angle of $e_s(\ii),e_{ss}(\jj)$ is uniformly separated away from zero for every $\ii,\jj\in\Sigma$, there exists a constant $c>0$ that for every $\y\in\Lambda\times\Sigma^+$ and $r>0$
\begin{equation}\label{ecomptransballs}
B^T_{c^{-1}r}(\xv,\ii)\subseteq B^t_{r}(\xv,\ii)\subseteq B^T_{cr}(\xv,\ii).
\end{equation}

\begin{lemma}\label{ltoLY2}
For any $\y=(\xv,\ii)\in\Lambda\times\Sigma^+$ with $\xv=\pi_-(\dots,i_{-1})$
$$\mu(B^T_{\delta}(\y)\cap\mathcal{P}(\y))=\mu\left(B^T_{\|A_{i_{-1}}|e_s(F^{-1}(\y))\|^{-1}\delta}(F^{-1}(\y))\right)p_{i_{-1}},$$
where $(p_1,\dots,p_N)$ is probability vector corresponding to $\mu$.
\end{lemma}

\begin{proof}
Since the directions $e_s$ are $F$-invariant, we get for any $\y'=(\xv',\ii')$ and $\infty>\delta'>0$
$$F\left(B^T_{\delta'}(\y')\times[i_{0}']\right)=\left(\left(B^T_{\|A_{i_0'}|e_s(\y')\|\delta'}(F(\y'))\right)\cap\mathcal{P}(F(\y'))\right)\times\Sigma^+.$$
The map $F$ is invertible, hence
$$B^T_{\delta'}(\y')\times[i_{0}']=F^{-1}\left(\left(\left(B^T_{\|A_{i_0'}|e_s(\y')\|\delta'}(F(\y'))\right)\cap\mathcal{P}(F(\y'))\right)\times\Sigma^+\right).$$
By taking $\y=F(\y')$ we have $\|A_{i_0'}|e_s(\y')\|=\|A_{i_{-1}}|e_s(F^{-1}(\y))\|$ and by taking\\ $\delta=\|A_{i_{-1}}|e_s(F^{-1}(\y))\|\delta'$
$$B^T_{\|A_{i_{-1}}|e_s(F^{-1}(\y))\|^{-1}\delta}(F^{-1}(\y))\times[i_{-1}]=F^{-1}\left(\left(B^T_{\delta}(\y)\cap\mathcal{P}(\y)\right)\times\Sigma^+\right).$$

The measure $\widehat{\mu}$ is $F$-invariant, therefore
\begin{multline*}
\mu(B^T_{\delta}(\y)\cap\mathcal{P}(\y))=\widehat{\mu}(\left(B^T_{\delta}(\y)\cap\mathcal{P}(\y)\right)\times\Sigma^+)=\widehat{\mu}(F^{-1}\left(\left(B^T_{\delta}(\y)\cap\mathcal{P}(\y)\right)\times\Sigma^+\right))=\\
\widehat{\mu}(B^T_{\|A_{i_{-1}}|e_s(F^{-1}(\y))\|^{-1}\delta}(F^{-1}(\y))\times[i_{-1}])=\mu(B^T_{\|A_{i_{-1}}|e_s(F^{-1}(\y))\|^{-1}\delta}(F^{-1}(\y)))p_{i_{-1}}.
\end{multline*}
\end{proof}

\begin{lemma}\label{lmultiptoLY2}
For every $\y=(\xv,\ii)\in\Lambda\times\Sigma^+$ and $n\geq2$
$$\|A^{(n)}(\ii)|e_s(\y)\|=\|A(\sigma^{n-1}\ii)|e_s(F^{n-1}(\y))\|\cdot\|A^{(n-1)}(\ii)|e_s(\y)\|.$$
\end{lemma}

\begin{proof}
By definition $\|A(\ii)|e_s(\y)\|=\sup_{v\in e_s(\y)}\frac{\|A(\ii)v\|}{\|v\|}$. On the other hand for every $v_1,v_2\in e_s(\y)$ there exists a constant $c\in\R$ such that $cv_1=v_2$. Therefore
$\|A(\ii)|e_s(\y)\|=\|A(\ii)v\|$ with any vector $v\in e_s(\y)$ with $\|v\|=1$. Hence,
\begin{multline*}
\|A^{(n)}(\ii)|e_s(\y)\|=\|A^{(n)}(\ii)v\|=\left\|A(\sigma^{n-1}\ii)\frac{A^{(n-1)}(\ii)v}{\|A^{(n-1)}(\ii)v\|}\right\|\|A^{(n-1)}(\ii)v\|=\\
\|A(\sigma^{n-1}\ii)|e_s(F^{n-1}(\y))\|\|A^{(n-1)}(\ii)|e_s(\y)\|,
\end{multline*} 
where we used in the last equation that $A^{(n-1)}(\ii)v\in e_s(F^{n-1}(\y))$, see Lemma~\ref{ldomsplit1}.
\end{proof}

Let us define functions $g(\y):=\cm{\y}{\xi^{ss}}(\mathcal{P}(\y))$ and $g_{\delta}(\y):=\frac{\mu(B^T_{\delta}(\y)\cap\mathcal{P}(\y))}{\mu(B^T_{\delta}(\y))}$. By definition and \eqref{establecondmeasure}, $g_{\delta}\rightarrow g$ as $\delta\rightarrow0+$ for $\mu$-almost everywhere and, since $g_{\delta}$ is uniformly bounded, \eqref{econdmeasure} implies $g_{\delta}\rightarrow g$ in $L^1(\widehat{\mu})$ as $\delta\rightarrow0+$.

\begin{lemma}\label{ltoLY2b}
The function $\sup_{\delta>0}\left\{-\log g_{\delta}\right\}$ is in $L^1(\widehat{\mu})$.
\end{lemma}

\begin{proof}
To verify the statement of the lemma, it is enough to show that
$$\sum_{k=1}^{\infty}\widehat{\mu}\left\{\y: \inf_{\delta>0}g_{\delta}(\y)<e^{-k}\right\}<\infty.$$
By \eqref{establecondmeasure}
\begin{equation}\label{eseged}
\widehat{\mu}\left\{\y: \inf_{\delta>0}g_{\delta}(\y)<e^{-k}\right\}=\sum_{i=1}^N\int\mu\left\{\xv\in f_i(\Lambda):\inf_{\delta>0}\frac{\mu(B^T_{\delta}(\xv,\ii)\cap f_i(\Lambda))}{\mu(B^T_{\delta}(\xv,\ii))}<e^{-k}\right\}d\nu(\ii).
\end{equation}
For a fixed $\ii\in\Sigma^+$ denote by $E_{k,\ii}^i$ the set $\left\{\xv\in f_i(\Lambda):\inf_{\delta>0}\frac{\mu(B^T_{\delta}(\xv,\ii)\cap f_i(\Lambda))}{\mu(B^T_{\delta}(\xv,\ii))}<e^{-k}\right\}$. Let $$\mathcal{E}_{k,\ii}^i:=\left\{B_{\delta}^T(\xv,\ii):\frac{\mu(B^T_{\delta}(\xv,\ii)\cap f_i(\Lambda))}{\mu(B^T_{\delta}(\xv,\ii))}<e^{-k}\right\}$$ be the corresponding collection of closed transversal balls. Then $\mathcal{E}_{k,\ii}^i$ is a cover of $E_{k,\ii}^i$, so by the Besicovitch Covering Theorem there exists a constant $c>0$ independent of $\ii$, $i$ and $k$, such that there are $c$ countable families of balls $\mathcal{F}_n$, $n=1,\dots,c$ with $\bigcup_{n=1}^c\mathcal{F}_n\subseteq\mathcal{E}_{k,\ii}^i$, such that $$E_{k,\ii}^i\subseteq\bigcup_{n=1}^c\bigcup_{B\in\mathcal{F}_n}B\text{\quad and\quad }B'\cap B''=\emptyset\text{ if }B',B''\in\mathcal{F}_j\text{ for }j=1,\dots,n.$$
Hence,
$$\mu\left\{\xv\in f_i(\Lambda):\inf_{\delta>0}\frac{\mu(B^T_{\delta}(\xv,\ii)\cap f_i(\Lambda))}{\mu(B^T_{\delta}(\xv,\ii))}<e^{-k}\right\}\leq\sum_{n=1}^c\sum_{B\in\mathcal{F}_n}\mu(B\cap f_i(\Lambda))\leq e^{-k}\sum_{n=1}^c\sum_{B\in\mathcal{F}_n}\mu(B)\leq ce^{-k}.$$
Therefore, by \eqref{eseged}
$$
\sum_{k=1}^{\infty}\widehat{\mu}\left\{\y: \inf_{\delta>0}g_{\delta}(\y)<e^{-k}\right\}\leq \sum_{k=1}^{\infty}cNe^{-k}<\infty.
$$
\end{proof}

The proof of the Proposition~\ref{pLY2} uses a slight modification of the result of Maker~\cite{M}.

\begin{lemma}[Maker, \cite{M}]\label{lmaker}
	Let $T:X\mapsto X$ be an endomorphism on $X\subset\R^d$ compact set and let $m$ be a $T$-invariant ergodic measure. Moreover, let $h_{n,k}:X\mapsto\R$ be a family of functions s.t. $\sup_{n,k}h_{n,k}\in L^1(m)$ and $\lim_{n-k\rightarrow\infty}h_{n,k}=h$ in $L^1(m)$ sense and $m$-almost everywhere, where $h\in L^1(m)$. Then
	\[
	\lim_{n\rightarrow\infty}\frac{1}{n}\sum_{k=0}^{n-1}h_{n,k}(T^kx)=\int h(x)dm(x)\text{ for $m$-a.e $x\in X$.}
	\]
\end{lemma}

\begin{proof}[Proof of Proposition~\ref{pLY2}]
By the definition of the transversal measure, the statement of the proposition is equivalent to
\begin{equation*}
\lim_{\delta\rightarrow0+}\frac{\log\mu(B^t_{\delta}(\y))}{\log\delta}=\frac{h_{\nu}-H(F\xi^{ss}|\xi^{ss})}{\ly{s}}\text{ for $\widehat{\mu}$-a.e $\y$.}
\end{equation*}
Hence, by \eqref{ecomptransballs} it is enough to show that
\begin{equation*}
\lim_{\delta\rightarrow0+}\frac{\log\mu(B^T_{\delta}(\y))}{\log\delta}=\frac{h_{\nu}-H(F\xi^{ss}|\xi^{ss})}{\ly{s}}\text{ for $\widehat{\mu}$-a.e $\y$.}
\end{equation*}
By Lemma~\ref{ldomsplit2}, it is sufficient to show that
\begin{equation}\label{elocdimLY2}
\lim_{n\rightarrow\infty}\frac{\log\mu\left(B^T_{\|A_{i_{-1}}\cdots A_{i_{-n}}|e_s(F^{-n}(\y))\|}(\y)\right)}{\log\alpha_1(A_{i_{-1}}\cdots A_{i_{-n}})}=\frac{h_{\nu}-H(F\xi^{ss}|\xi^{ss})}{\ly{s}}\text{ for $\widehat{\mu}$-a.e $\y$,}
\end{equation}
where $\y=(\xv,\ii)\in\Lambda\times\Sigma^+$ with $\xv=\pi_-(\dots,i_{-2},i_{-1})$. We write the measure of the ball as
\begin{multline*}
\mu\left(B^T_{\|A_{i_{-1}}\cdots A_{i_{-n}}|e_s(F^{-n}(\y))\|}(\y)\right)=\\
\mu(B^T_{1}(F^{-n}(\y)))\frac{\mu(B^T_{\|A_{i_{-n}}|e_s(F^{-n}(\y))\|}(F^{-n+1}(\y)))}{\mu(B^T_{1}(F^{-n}(\y)))}\prod_{k=1}^{n-1}
\frac{\mu\left(B^T_{\|A_{i_{-k}}\cdots A_{i_{-n}}|e_s(F^{-n}(\y))\|}(F^{-k+1}(\y))\right)}{\mu\left(B^T_{\|A_{i_{-k-1}}\cdots A_{i_{-n}}|e_s(F^{-n}(\y))\|}(F^{-k}(\y))\right)}.
\end{multline*}
Applying Lemma~\ref{ltoLY2} and Lemma~\ref{lmultiptoLY2} we get for every $k=1,\dots,n$
\begin{multline*}
\mu\left(B^T_{\|A_{i_{-k-1}}\cdots A_{i_{-n}}|e_s(F^{-n}(\y))\|}(F^{-k}(\y))\right)=\\
\mu\left(B^T_{\|A_{i_{-k}}|e_s(F^{-k}(\y))\|^{-1}\|A_{i_{-k}}\cdots A_{i_{-n}}|e_s(F^{-n}(\y))\|}(F^{-k}(\y))\right)=\\
\mu\left(B^T_{\|A_{i_{-k}}\cdots A_{i_{-n}}|e_s(F^{-n}(\y))\|}(F^{-k+1}(\y))\cap\mathcal{P}(F^{-k+1}(\y))\right)p_{i_{-k}}^{-1}.
\end{multline*}
Hence,
\begin{multline*}
\frac{1}{n}\log\mu\left(B^T_{\|A_{i_{-1}}\cdots A_{i_{-n}}|e_s(F^{-n}(\y))\|}(\y)\right)=\\
\frac{1}{n}\log\mu(B^T_{1}(F^{-n}(\y)))-\frac{1}{n}\sum_{k=1}^n\log g_{\|A_{i_{-k}}\cdots A_{i_{-n}}|e_s(F^{-n}(\y))\|}(F^{-k+1}(\y))+\frac{1}{n}\sum_{k=1}^n\log p_{i_{-k}}.
\end{multline*}
Let us define a function $h_{n,k}(\y):=\log g_{\|A_{i_{-1}}\cdots A_{i_{-n+k-1}}|e_s(F^{-n+k-1}(\y))\|}(\y)$. Then
\begin{multline*}
\frac{1}{n}\log\mu\left(B^T_{\|A_{i_{-1}}\cdots A_{i_{-n}}|e_s(F^{-n}(\y))\|}(\y)\right)=\\
\frac{1}{n}\log\mu(B^T_{1}(F^{-n}(\y)))-\frac{1}{n}\sum_{k=1}^nh_{n,k}(F^{-k+1}(\y))+\frac{1}{n}\sum_{k=1}^n\log p_{i_{-k}}.
\end{multline*}
Since $\|A_{i_{-1}}\cdots A_{i_{-n+k-1}}|e_s(F^{-n+k-1}(\y))\|\rightarrow0$ uniformly on $\Lambda\times\Sigma^+$ as $n-k\rightarrow\infty$, $\lim_{n-k\rightarrow\infty}h_{n,k}=\log g\text{ in $L^1(\widehat{\mu})$ and $\widehat{\mu}$-almost everywhere.}$ Moreover, by Lemma~\ref{ltoLY2b} we can apply Maker's ergodic theorem Lemma~\ref{lmaker}. Thus, by \eqref{econdenteq}
\begin{multline*}
\lim_{n\rightarrow\infty}\frac{1}{n}\sum_{k=1}^nh_{n,k}(F^{-k+1}(\y))=\int\log g(\y)d\widehat{\mu}(\y)\\
=\int\log\cm{\y}{\xi^{ss}}(\mathcal{P}(\y))d\widehat{\mu}(\y)=-H(\mathcal{P}|\xi^{ss})=-H(F\xi^{ss}|\xi^{ss})\text{ for $\widehat{\mu}$-a.e $\y$.}
\end{multline*}
On the other hand
\begin{eqnarray*}
&& \lim_{n\rightarrow\infty}\frac{1}{n}\log\mu(B^T_{1}(F^{-n}(\y)))=0\text{ for every $\y\in\Lambda\times\Sigma^+$ and}\\
&& \lim_{n\rightarrow\infty}\frac{1}{n}\sum_{k=1}^n\log p_{i_{-k}}=\sum_{i=1}^Np_i\log p_i=-h_{\nu}\text{ for $\widehat{\mu}$-a.e $\y\in\Lambda\times\Sigma^+$,}
\end{eqnarray*}
which implies
\begin{equation}\label{eentLY2}
\lim_{n\rightarrow\infty}\frac{1}{n}\log\mu\left(B^T_{\|A_{i_{-1}}\cdots A_{i_{-n}}|e_s(F^{-n}(\y))\|}(\y)\right)=-h_{\nu}+H(F\xi^{ss}|\xi^{ss})
\end{equation}
Applying Oseledec's Theorem, we have
$$\lim_{n\to\infty}\frac{1}{n}\log\alpha_1(A_{i_{-1}}\cdots A_{i_{-n}})=-\ly{s}\text{ for $\nu$-a.e $\ii$,}$$
which together with the equation \eqref{eentLY2} implies \eqref{elocdimLY2}.
\end{proof}


Let $B^s_r(\xv,\ii)$ denote the square on $\Lambda\times\left\{\ii\right\}$ with a side parallel to $e_{ss}(\ii)$ and length $2r$ centered at $(\xv,\ii)\in\Lambda\times\Sigma^+$. It is easy to see that there exists a constant $c>0$ that for every $(\xv,\ii)\in\Lambda\times\Sigma^+$, \begin{equation}\label{ecompballs}B_{c^{-1}r}(\xv)\subseteq B^s_r(\xv,\ii)\subseteq B_{cr}(\xv),\end{equation} where $B_r(\xv)$ is the usual Euclidean ball on $\R^2$.

\begin{prop}\label{pLY3}
For $\mu$-a.e $\xv\in\Lambda$
$$\liminf_{r\rightarrow0+}\frac{\log\mu(B_r(\xv))}{\log r}\geq\frac{H(F\xi^{ss}|\xi^{ss})}{\ly{ss}}+\frac{h_{\nu}-H(F\xi^{ss}|\xi^{ss})}{\ly{s}}.$$
\end{prop}

\begin{proof}
By \eqref{ecompballs}, it is enough to show that
\begin{equation}\label{eLY3}
\liminf_{r\rightarrow0+}\frac{\log\mu(B_r^s(\y))}{\log r}\geq\frac{H(F\xi^{ss}|\xi^{ss})}{\ly{ss}}+\frac{h_{\nu}-H(F\xi^{ss}|\xi^{ss})}{\ly{s}}\text{ for $\widehat{\mu}$-a.e $\y$.}
\end{equation}
For simplicity, let $d_s:=\frac{H(F\xi^{ss}|\xi^{ss})}{\ly{ss}}$. By Proposition~\ref{pLY1}, the measure $\cm{\y}{\xi^{ss}}$ is exact dimensional and by Egorov's Theorem for every $\varepsilon>0$ there exists a set $J_1\subseteq\Lambda\times\Sigma^+$ with $\widehat{\mu}(J_1)>1-\varepsilon$ such that there exists a $M_1>0$ that for every $m\geq M_1$ and $\y\in J_1$
$$
\cm{\y}{\xi^{ss}}(B^{ss}_{2e^{-m}}(\y))\leq e^{m(-d_s+\varepsilon)}.
$$
By the definition of $B^s_r(\y)$ it is easy to see that $B^s_r(\y)\cap \xi^{ss}(\y)=B^{ss}_r(\y)$ and by the definition of conditional measures
\begin{equation}\label{etLY3}
\cm{\z}{\xi^{ss}}(B^{s}_{2e^{-m}}(\y))\leq e^{m(-d_s+\varepsilon)}\text{ for every $\y\in J_1$ and $\xi^{ss}(\underline{\mathbf{z}})=\xi^{ss}(\y)$.}
\end{equation}
The combination of Lebesgue density Theorem and Egorov's Theorem implies that there exists a set $J_2\subseteq J_1$ with $\widehat{\mu}(J_2)>1-2\varepsilon$ and $M_2>0$ such that for every $m\geq M_2$ and $\y\in J_2$
$$\mu(J_1\cap B^s_{e^{-m}}(\y))\geq\frac{1}{2}\mu(B^s_{e^{-m}}(\y)).$$
By \eqref{etLY3}, for every $\z\in B^t_{e^{-m}}(\y)$ such that there exists a $\z'\in\xi^{ss}(\z)\cap B^s_{e^{-m}}(\y)\cap J_1$
$$\cm{\z}{\xi^{ss}}(B^s_{e^{-m}}(\y)\cap J_1)\leq\cm{\z}{\xi^{ss}}(B^s_{2e^{-m}}(\z')\cap J_1)\leq e^{m(-d_s+\varepsilon)}.$$
If $\xi^{ss}(\z)\cap B^s_{e^{-m}}(\y)\cap J_1=\emptyset$ then the bound above is trivial. Hence, for every $\y\in J_2$
$$
\mu(B^s_{e^{-m}}(\y))\leq2\int_{B^t_{e^{-m}}(\y)}\cm{\z}{\xi^{ss}}(B^s_{e^{-m}}(\y)\cap J_1)d\widehat{\mu}(\z)\leq2e^{m(-d_s+\varepsilon)}\mu(B^t_{e^{-m}}(\y)).
$$
Since $\varepsilon>0$ was arbitrary, inequality \eqref{eLY3} follows by Proposition~\ref{pLY2}.
\end{proof}

\begin{prop}\label{pLY4}
For $\mu$-a.e $\xv\in\Lambda$
$$\limsup_{r\rightarrow0+}\frac{\log\mu(B_r(\xv))}{\log r}\leq\frac{H(F\xi^{ss}|\xi^{ss})}{\ly{ss}}+\frac{h_{\nu}-H(F\xi^{ss}|\xi^{ss})}{\ly{s}}.$$
\end{prop}

\begin{proof}
For simplicity, let $h_s:=H(F\xi^{ss}|\xi^{ss})=d_s\ly{ss}$. We remind the reader that $\widehat{\mu}=\mu\times\nu$. By applying Egorov's Theorem for Proposition~\ref{pLY1} and for the Shannon-McMillan-Breiman Theorem, we get that for every $\varepsilon>0$ there exists a set $J_1$ with $\widehat{\mu}(J_1)>1-\varepsilon$ and $M_1>0$ such that for every $\y=(\xv,\ii)\in J_1$ and every $m\geq M_1$
\begin{eqnarray}
&& \cm{\y}{\xi^{ss}}(B_{\frac{\kappa}{2}e^{-m(\ly{s}-2\varepsilon)}}(\xv)\times\left\{\ii\right\})\geq e^{-m(\ly{s}-2\varepsilon)(d_s+\varepsilon)},\label{etoLY41}\\
&& \cm{\y}{\xi^{ss}}\left(\left(\bigvee_{k=0}^{m-1}F^k\mathcal{P}\right)(\y)\right)\leq e^{-m(h_s-\varepsilon)},\label{etoLY42}\\
&& \left(\bigvee_{k=0}^{m-1}F^k\mathcal{P}\right)(\y)\subseteq B_{e^{-m(\ly{s}-2\varepsilon)}}(\xv)\times\Sigma^+,\label{etoLY43}\\
&& \widehat{\mu}\left(\left(\bigvee_{k=0}^{m-1}F^k\mathcal{P}\right)(\y)\right)\geq e^{-m(h_{\nu}+\varepsilon)}\label{etoLY44},
\end{eqnarray}
where $\kappa=\min_{i\neq j}\mathrm{dist}(f_i(\Lambda),f_j(\Lambda))>0$. Applying Lebesgue's density Theorem and Egorov's Theorem, there exists a set $J_2\subseteq J_1$ with $\widehat{\mu}(J_2)>1-2\varepsilon$ and $M_2\geq M_1$ such that for every $\y=(\xv,\ii)\in J_2$ and every $m\geq M_2$
\begin{equation}\label{etoLY46}
\dfrac{ \cm{\y}{\xi^{ss}}\left(\left(B_{\frac{\kappa}{2}e^{-m(\ly{s}-2\varepsilon)}}(\xv)\times\left\{\ii\right\}\right)\cap J_1\right)}{ \cm{\y}{\xi^{ss}}(B_{\frac{\kappa}{2}e^{-m(\ly{s}-2\varepsilon)}}(\xv)\times\left\{\ii\right\})}\geq\frac{1}{2}.
\end{equation}
For every $m\geq M_2$ we can define a finite set $\left\{\y_i\right\}_{i=1}^{N^m}$ such that for any $i\neq j$ $$\left(\bigvee_{k=0}^{m-1}F^k\mathcal{P}\right)(\y_i)\cap\left(\bigvee_{k=0}^{m-1}F^k\mathcal{P}\right)(\y_j)=\emptyset,$$
and $\y_i\in J_1$ whenever $J_1\cap\left(\bigvee_{k=0}^{m-1}F^k\mathcal{P}\right)(\y_i)\neq\emptyset$. For simplicity, we introduce the notation $L_m(\y):=B_{\frac{\kappa}{2}e^{-m(\ly{s}-2\varepsilon)}}(\y)\cap J_1$. By \eqref{etoLY41} and \eqref{etoLY46}, for any $\y\in J_2$
\begin{equation}\label{etoLY47}
\cm{\y}{\xi^{ss}}(L_m(\y))\geq\frac{1}{2}\cm{\y}{\xi^{ss}}(B_{\frac{\kappa}{2}e^{-m(\ly{s}-2\varepsilon)}}(\y))\geq\frac{1}{2}e^{-m(\ly{s}-2\varepsilon)(d_s+\varepsilon)}.
\end{equation}

Then by \eqref{etoLY42}
\begin{multline*}
\cm{\y}{\xi^{ss}}(L_m(\y))\leq\sum_{i:\y_i\in J_1}\cm{\y}{\xi^{ss}}\left(\left(\bigvee_{k=0}^{m-1}F^k\mathcal{P}\right)(\y_i)\cap L_m(\y)\right)\\
\leq \sharp\left\{\y_i\in J_1:L_m(\y)\cap\left(\bigvee_{k=0}^{m-1}F^k\mathcal{P}\right)(\y_i)\neq\emptyset\right\}\cdot\max_{i:\y_i\in J_1}\left\{\cm{\y}{\xi^{ss}}\left(\left(\bigvee_{k=0}^{m-1}F^k\mathcal{P}\right)(\y_i)\right)\right\}\\
\leq\sharp\left\{\y_i\in J_1:L_m(\y)\cap\left(\bigvee_{k=0}^{m-1}F^k\mathcal{P}\right)(\y_i)\neq\emptyset\right\}e^{-m(h_s-\varepsilon)}.
\end{multline*}
Hence,
\begin{equation}\label{etoLY45}
\cm{\y}{\xi^{ss}}(L_m(\y))e^{m(h_s-\varepsilon)}\leq\sharp\left\{\y_i\in J_1:L_m(\y)\cap\left(\bigvee_{k=0}^{m-1}F^k\mathcal{P}\right)(\y_i)\neq\emptyset\right\}.
\end{equation}
On the other hand, if $L_m(\y)\cap\left(\bigvee_{k=0}^{m-1}F^k\mathcal{P}\right)(\y_i)\neq\emptyset$ then by \eqref{etoLY43}
$$
\left(\bigvee_{k=0}^{m-1}F^k\mathcal{P}\right)(\y_i)\subseteq B_{2e^{-m(\ly{s}-2\varepsilon)}}(\xv)\times\Sigma^+.
$$
Therefore,
\begin{multline*}
\mu(B_{2e^{-m(\ly{s}-2\varepsilon)}}(\xv))=\widehat{\mu}(B_{2e^{-m(\ly{s}-2\varepsilon)}}(\xv)\times\Sigma^+)\\
\geq\sharp\left\{\y_i\in J_1:L_m(\y)\cap\left(\bigvee_{k=0}^{m-1}F^k\mathcal{P}\right)(\y_i)\neq\emptyset\right\}\min_{i:\y_i\in J_1}\left\{\widehat{\mu}\left(\left(\bigvee_{k=0}^{m-1}F^k\mathcal{P}\right)(\y_i)\right)\right\}.
\end{multline*}
By \eqref{etoLY44} and \eqref{etoLY45}, for any $\y\in J_2$
$$
\mu(B_{2e^{-m(\ly{s}-2\varepsilon)}}(\xv))\geq\cm{\y}{\xi^{ss}}(L_m(\y))e^{m(h_s-\varepsilon)}e^{-m(h_{\nu}+\varepsilon)}.
$$
Using \eqref{etoLY47}, for any $\y=(\xv,\ii)\in J_2$
$$
\mu(B_{2e^{-m(\ly{s}-2\varepsilon)}}(\xv))\geq\frac{1}{2}e^{-m(\ly{s}-2\varepsilon)(d_s+\varepsilon)}e^{m(h_s-\varepsilon)}e^{-m(h_{\nu}+\varepsilon)}.
$$
Hence, for any $\y\in J_2$
$$
\limsup_{m\rightarrow\infty}\frac{\log\mu(B_{2e^{-m(\ly{s}-2\varepsilon)}}(\xv))}{-m(\ly{s}-2\varepsilon)}\leq d_s+\varepsilon+\frac{h_{\nu}-h_s}{\ly{s}-2\varepsilon}+\varepsilon+\frac{2\varepsilon}{\ly{s}-2\varepsilon}
$$
Since $\varepsilon>0$ was arbitrary, the statement of the proposition follows.
\end{proof}

\begin{proof}[Proof of Theorem~\ref{tmain1}]
Proposition~\ref{pLY3} and Proposition~\ref{pLY4} together imply that $\mu$ is exact dimensional, moreover
$$
\dim_H\mu=\frac{H(F\xi^{ss}|\xi^{ss})}{\ly{ss}}+\frac{h_{\nu}-H(F\xi^{ss}|\xi^{ss})}{\ly{s}}.
$$
Simple algebraic manipulations show that
$$
\frac{H(F\xi^{ss}|\xi^{ss})}{\ly{ss}}+\frac{h_{\nu}-H(F\xi^{ss}|\xi^{ss})}{\ly{s}}=\frac{h_{\nu}}{\ly{ss}}+\left(1-\frac{\ly{s}}{\ly{ss}}\right)\frac{h_{\nu}-H(F\xi^{ss}|\xi^{ss})}{\ly{s}}.
$$
The proof can be finished by applying Proposition~\ref{pLY2}.
\end{proof}

\section{Applications}\label{sappl}

\subsection{Hueter-Lalley Theorem} This section is devoted to showing some applications of our main theorems. In the point of view of Theorem~\ref{tmain1b}, to prove that the Hausdorff dimension of a self-affine measure is equal to its Lyapunov dimension, one has to study the dimension of $\nu_{ss}$ defined in \eqref{edirections}. The measure $\nu_{ss}$ is basically a self-conformal measure associated to an IFS on the projective space. If the IFS on the projective space satisfies some separation condition then one may be able to calculate its dimension. Basically, Hueter and Lalley \cite{HL} used this phenomena to prove their theorem. Now we reprove their result.

\begin{theorem}\label{tHL}
Let $\Alpha=\left\{A_1, A_2,\dots, A_N\right\}$ be a finite set of contracting, non-singular $2\times2$ matrices, and let $\Phi=\left\{f_i(\xv)=A_i\xv+\tv_i\right\}_{i=1}^N$ be an iterated function system on the plane with affine mappings. Let $\nu$ be a left-shift invariant and ergodic Bernoulli-probability measure on $\Sigma^{+}$, and $\mu$ be the corresponding self-affine measure. Assume that
\begin{enumerate}
    \item\label{chl1} $\Alpha$ satisfies dominated splitting,
    \item\label{chl2} $\Alpha$ satisfies the \textnormal{backward non-overlapping condition}, i.e. there exists a backward invariant multicone $M$ such that $A^{-1}_i(M^{o})\subseteq M^{o}$ and $A^{-1}_i(M^{o})\cap A_j^{-1}(M^{o})=\emptyset$ for every $i\neq j$,
    \item\label{chl3} $\Alpha$ satisfies the \textnormal{$1$-bunched property}, i.e. for every $i=1,\dots,N$ $\alpha_1(A_i)^2\leq\alpha_2(A_i)$,
    \item\label{chl4} $\Phi$ satisfies the strong separation condition.
\end{enumerate}
Then
\begin{equation*}
\dim_H\mu=\ldim\mu=\frac{h_{\nu}}{\ly{s}}\leq1.
\end{equation*}
\end{theorem}

The proof of the theorem uses the following lemma.

\begin{lemma}\label{ltoHL}
Let $\Alpha=\left\{A_1, A_2,\dots, A_N\right\}$ be a finite set of contracting, non-singular $2\times2$ matrices and let $\nu$ be a left-shift invariant and ergodic Bernoulli-probability measure on $\Sigma^{+}$. Assume that
\begin{enumerate}
    \item $\Alpha$ satisfies dominated splitting,
    \item $\Alpha$ satisfies the backward non-overlapping condition.
\end{enumerate}
Let $e_{ss}:\Sigma^+\mapsto\mathbf{P}^1$ be the projection defined in Lemma~\ref{ldomsplit3b}. Then
$$
\dim_H\nu_{ss}=\dim_H\nu\circ e_{ss}^{-1}=\frac{h_{\nu}}{\ly{ss}-\ly{s}},
$$
where $\ly{ss}$ and $\ly{s}$ are the Lyapunov exponents defined in Lemma~\ref{ldomsplit2}.
\end{lemma}

\begin{proof}
The projective space $\mathbf{P}^1$ is equivalent to the upper half unit sphere in $\R^2$. We define an iterated function system on $\mathbf{P}^1$ by $\Alpha$ in the natural way, i.e.
$$
\widetilde{A}_i:\theta\in\mathbf{P}^1\mapsto\sgn((A_i^{-1}\theta)_2)\frac{A_i^{-1}\theta}{\|A_i^{-1}\theta\|},
$$ where $\sgn((A_i^{-1}\theta)_2)$ denotes the signum of the second coordinate of the vector $A_i^{-1}\theta$. By \cite[Lemma~3.2]{BR}, the IFS $\widetilde{\Alpha}=\left\{\widetilde{A}_1,\dots,\widetilde{A}_N\right\}$ is uniformly contracting on $M$, where $M$ is the backward invariant multicone with non-overlapping condition. Hence, the measure $\nu_{ss}$ is the invariant measure associated to the IFS $\widetilde{\Alpha}$, and
\[
\dim_H\nu_{ss}=\lim_{r\rightarrow0+}\frac{\log\nu_{ss}(B_r(\theta))}{\log r}\text{ for $\nu_{ss}$-a.e $\theta$,}
\]
where $B_r(\theta)$ denotes the ball with radius $r$ centered at $\theta$ according to the spherical distance. Since $\Alpha$ satisfies the backward non-overlapping condition
\[
\dim_H\nu_{ss}=\lim_{n\rightarrow\infty}\frac{\log\nu_{ss}(\widetilde{A}_{i_1}\circ\cdots\widetilde{A}_{i_n}(M))}{\log \diam(\widetilde{A}_{i_1}\circ\cdots\widetilde{A}_{i_n}(M))}\text{ for $\nu$-a.e $\ii$,}
\]
where $\diam(.)$ denotes the diameter of a set according to the spherical distance. It is easy to see that for any $\theta_1,\theta_2\in\mathbf{P}^1$, and any $\underline{0}\neq\underline{v}\in\theta_1$, $\underline{0}\neq\underline{w}\in\theta_2$
$$
\frac{\|\underline{v}\times\underline{w}\|}{\|\underline{v}\|\|\underline{w}\|}\leq d(\theta_1,\theta_2)\leq\frac{2\|\underline{v}\times\underline{w}\|}{\|\underline{v}\|\|\underline{w}\|},
$$
where $\underline{v}\times\underline{w}$ denotes the standard vector product.
Thus,
\[
\frac{\det(A_{\underline{i}}^{-1})}{\|A_{\underline{i}}^{-1}|\theta_1\|\|A_{\underline{i}}^{-1}|\theta_2\|}\frac{d(\theta_1,\theta_2)}{2}\leq d(\widetilde{A}_{\underline{i}}(\theta_1),\widetilde{A}_{\underline{i}}(\theta_2))\leq \frac{2\det(A_{\underline{i}}^{-1})}{\|A_{\underline{i}}^{-1}|\theta_1\|\|A_{\underline{i}}^{-1}|\theta_2\|}
\]
for any $\underline{i}\in\Sigma^*$. Since every $\theta\in M$ is uniformly separated away from the stable directions, we get
$$
\lim_{n\rightarrow\infty}\frac{1}{n}\log\diam(\widetilde{A}_{i_1}\circ\cdots\widetilde{A}_{i_n}(M))=\ly{s}-\ly{ss}\text{ for $\nu$-a.e $\ii$.}
$$
On the other hand $$\lim_{n\rightarrow\infty}\frac{1}{n}\log\nu_{ss}(\widetilde{A}_{i_1}\circ\cdots\widetilde{A}_{i_n}(M))=\lim_{n\rightarrow\infty}\frac{1}{n}\log\nu([i_1,\dots,i_n])=-h_{\nu}\text{ for $\nu$-a.e $\ii$.}$$
The statement follows by taking the ratio of the previous two limits.
\end{proof}

Now, we show a modification of Marstrand's projection theorem~\cite{M}. Kaufman~\cite{K} showed an upper bound on the exceptional set of directions, where the dimension drops. We use in the next lemma the method of Kaufman~\cite{K}, however, for later usage we need a better lower bound on the dimension of projected measure, therefore for the comfortability of the reader, we prove it here.

\begin{lemma}\label{lmarstrand}
	Let $m$ be a probability measure on $\R^2$ and let $\lambda$ be a measure on $[0,\pi)$. For a $\theta\in[0,\pi)$ denote by $\proj_{\theta}$ the orthogonal projection onto the line perpendicular to the vector $(\sin\theta,\cos\theta)$. Then for every $\varepsilon>0$ there exists a set $A_{\varepsilon}\subseteq[0,\pi)$ such that $\lambda(A_{\varepsilon})>0$ and
	\begin{equation}\label{eprojdim}
	\dim_H(\proj_{\theta})_*m\geq\min\left\{1,\dim_Hm,\dim_H\lambda\right\}-\varepsilon\text{ for every $\theta\in A_{\varepsilon}$.}
	\end{equation}
\end{lemma}

\begin{proof}
	Let us denote $\min\left\{1,\dim_Hm,\dim_H\lambda\right\}$ by $s$. Since $\dim_H\lambda\geq s$ then using \eqref{elocdimHdim} and Egorov's Theorem for every $\varepsilon>0$ there exists a set $A_{\varepsilon}$ and $C>0$ such that $\lambda(A_{\varepsilon})>0$ and for every $\theta\in A_{\varepsilon}$ and $r>0$
	$$
	\lambda((\theta-r,\theta+r))\leq Cr^{s-\varepsilon/2}.
	$$ Moreover, without loss of generality we may assume that $A_{\varepsilon}$ is bounded away from $0$ and $\pi$, i.e. there exists a constant $c>0$ s.t. $\mathrm{dist}(\theta,\left\{0,\pi\right\})>c$ for every $\theta\in A_{\varepsilon}$. Let $\lambda':=\left.\lambda\right|_{A_{\varepsilon}}/\lambda(A_{\varepsilon})$ be the restricted and normalized measure. It is easy to see that there exists a constant $c'>0$ such that for any interval $I\subseteq[0,\pi)$ $$\lambda'(I)\leq c'|I|^{s-\varepsilon/2}.$$
	We prove that for almost every point w.r.t $\lambda'$ \eqref{eprojdim} holds.
	
	On the other hand, by \eqref{elocdimHdim} for every $\varepsilon>0$ the exists a set $\Omega$ such that $m(\Omega)>0$ and $\underline{d}_{m}(\xv)>\dim_Hm-\varepsilon/4$. By Egorov's Theorem, there exists a set $\Omega'\subset\Omega$ and $R>0$ such that $m(\Omega')>0$ and $m(B_r(\xv))\leq r^{\dim_Hm-\varepsilon/2}$ for every $\xv\in\Omega'$ and $r<R$. Let $\widetilde{m}=\left.m\right|_{\Omega'}$. Thus, simple calculations show that
	\begin{equation}\label{ekell}
	\iint\frac{1}{\|\xv-\yv\|^{\dim_Hm-\varepsilon}}dm(\xv)d\widetilde{m}(\yv)\leq\int\sum_{n=0}^{\infty}2^{(\dim_Hm-\varepsilon)(n+1)}m(B_{2^{-n}}(\yv))d\widetilde{m}(\yv)<\infty.
	\end{equation}
	
	For simplicity we denote $(\proj_{\theta})_*m$ by $m_{\theta}$ and $(\proj_{\theta})_*\widetilde{m}$ by $\widetilde{m}_{\theta}$. Now we show that 
	\[
	\mathcal{I}:=\iiint\frac{1}{|x-y|^{s-\varepsilon}}dm_{\theta}(x)d\widetilde{m}_{\theta}(y)d\lambda'(\theta)<+\infty.
	\]
	
	Applying Fubini's Theorem we have
	\begin{multline*}
	\mathcal{I}=\iiint\frac{1}{|\proj_{\theta}(\xv)-\proj_{\theta}(\yv)|^{s-\varepsilon}}d\lambda'(\theta)dm(\xv)d\widetilde{m}(\yv)=\\
	\iint\frac{1}{\|\xv-\yv\|^{s-\varepsilon}}\int\dfrac{1}{\left(\frac{|\proj_{\theta}(\xv)-\proj_{\theta}(\yv)|}{\|\xv-\yv\|}\right)^{s-\varepsilon}}d\lambda'(\theta)dm(\xv)d\widetilde{m}(\yv).
	\end{multline*}
	Applying some algebraic manipulation we have for every $\xv\neq\yv\in\R^2$
	\[
	\int\dfrac{1}{\left(\frac{|\proj_{\theta}(\xv)-\proj_{\theta}(\yv)|}{\|\xv-\yv\|}\right)^{s-\varepsilon}}d\lambda'(\theta)\leq
	2^{s-\varepsilon}\sum_{n=0}^{\infty}2^{n(s-\varepsilon)}\lambda'\left(\left\{\theta:\frac{|\proj_{\theta}(\xv)-\proj_{\theta}(\yv)|}{\|\xv-\yv\|}\leq\frac{1}{2^n}\right\}\right).
	\]
	Since the set $\left\{\theta:\frac{|\proj_{\theta}(\xv)-\proj_{\theta}(\yv)|}{\|\xv-\yv\|}\leq\frac{1}{2^n}\right\}$ is contained in at most two intervals with length $c''/2^{n}$ we get
	\[
	\int\dfrac{1}{\left(\frac{|\proj_{\theta}(\xv)-\proj_{\theta}(\yv)|}{\|\xv-\yv\|}\right)^{s-\varepsilon}}d\lambda'(\theta)\leq2^{s-\varepsilon}\sum_{n=0}^{\infty}2^{n(s-\varepsilon)}\left(\frac{c''}{2^n}\right)^{s-\varepsilon/2}<+\infty
	\]
	Hence, by \eqref{ekell}
	$$
	\mathcal{I}\leq C\iint\frac{1}{\|\xv-\yv\|^{\dim_Hm-\varepsilon}}dm(\xv)d\widetilde{m}(\yv)<\infty.
	$$
	Therefore, $\iint\frac{1}{|x-y|^{s-\varepsilon}}d\widetilde{m}_{\theta}(x)d\widetilde{m}_{\theta}(y)<\infty$ for $\lambda'$-a.e. $\theta$. By Frostman's Lemma, $\dim_H\widetilde{m}_{\theta}\geq s-\varepsilon$ for $\lambda'$-a.e. $\theta$. But $\widetilde{m}_{\theta}\ll m_{\theta}$, thus $$
	\dim_Hm_{\theta}=\inf\left\{\dim_HA:m_{\theta}(A^c)=0\right\}\geq\inf\left\{\dim_HA:\widetilde{m}_{\theta}(A^c)=0\right\}=\dim_H\widetilde{m}_{\theta}\geq s-\varepsilon$$ for $\lambda'$-a.e. $\theta$, which had to be proven.
\end{proof}

\begin{proof}[Proof of Theorem~\ref{tHL}]
By $1$-bunched property, $\ly{ss}\leq2\ly{s}$. Hence, by using Lemma~\ref{ltoHL}
$$
1\geq\dim_H\nu_{ss}=\frac{h_{\nu}}{\ly{ss}-\ly{s}}\geq\frac{h_{\nu}}{\ly{s}}=\ldim\mu\geq\dim_H\mu.
$$
Thus, applying Theorem~\ref{tmain1b} we have that $\ldim\mu=\dim_H\mu$.
\end{proof}

\begin{cor}[Hueter-Lalley,\cite{HL}]
Let $\Alpha=\left\{A_1, A_2,\dots, A_N\right\}$ be a finite set of contracting, non-singular $2\times2$ matrices, and let $\Phi=\left\{f_i(\xv)=A_i\xv+\tv_i\right\}_{i=1}^N$ be an iterated function system on the plane with affine mappings and denote by $\Lambda$ the attractor of the IFS~$\Phi$. With the assumptions (\ref{chl1})-(\ref{chl4}) of Theorem~\ref{tHL}
\[\dim_H\Lambda=\dim_B\Lambda=s\leq1,
\]
where $s$ is the unique root of the pressure function $P(s)$, defined in \eqref{esap}.
\end{cor}

\begin{proof}
It is easy to see that the assumptions (\ref{chl1})-(\ref{chl4}) of Theorem~\ref{tHL} are inherited by the higher iterations, i.e. for any $n\geq1$ the IFS $\Phi^n=\left\{f_{\underline{i}}\right\}_{|\underline{i}|=n}$ and the set of matrices $\Alpha^n=\left\{A_{\underline{i}}\right\}_{|\underline{i}|=n}$ satisfy the assumptions (\ref{chl1})-(\ref{chl4}).

Let us define a monotone decreasing sequence $\left\{s_n\right\}_{n=1}^{\infty}$ such that $s_n$ are the unique solution of the equations
\[
\sum_{|\underline{i}|=n}\alpha_1(A_{\underline{i}})^{s_n}=1.
\]We define the left-shift invariant Bernoulli measure $\nu_n$ with probability vector $\left(\alpha_1(A_{\underline{i}})^{s_n}\right)_{|\underline{i}|=n}$ and let $\mu_n$ be the associated self-affine measure. Then by Theorem~\ref{tHL} and \eqref{elyapint}, for every $n\geq1$
\[
1\geq\dim_H\mu_n=\frac{h_{\nu_n}}{\chi_{\mu_n}^s}\geq\frac{s_n}{1+\frac{C}{n\log\alpha_{\max}}},
\]
where $\alpha_{\max}=\max_i\alpha_1(A_i)$. Hence $\lim_{n\rightarrow}s_n=s\leq1$. Moreover, by \cite[Proposition~5.1]{F}
\[
s\geq\overline{\dim}_B\Lambda\geq\dim_H\Lambda\geq\lim_{n\rightarrow\infty}\dim_H\mu_n=s.
\]\end{proof}

\subsection{Triangular matrices}
The other way to study the dimension of $\nu_{ss}$ is to handle the overlaps of the associated IFS on the projective space. Since this IFS is very difficult to handle in general, we focus on a special family of self-affine sets. Let us assume that the matrices in $\Alpha$ are lower triangular, i.e.
\begin{equation}\label{etrig}
A_i=\left[
      \begin{array}{cc}
        a_i & 0 \\
        b_i & c_i \\
      \end{array}
    \right],
\end{equation}
where $0<|a_i|,|c_i|<1$ for every $i=1,\dots,N$. Using \cite[Theorem~2.5]{FM}, the subadditive pressure function $P(s)$ defined in \eqref{esap} can be written in a simpler form, i.e
\begin{equation}\label{etrigsap}
P(s)=\left\{\begin{array}{cc}
       \log\max\left\{\sum_{i=1}^N|a_i|^{s},\sum_{i=1}^N|c_i|^{s}\right\} & \text{if $0\leq s<1$} \\
       \log\max\left\{\sum_{i=1}^N|a_i||c_i|^{s-1},\sum_{i=1}^N|c_i||a_i|^{s-1}\right\} & \text{if $1\leq s<2$} \\
       \log\sum_{i=1}^N(|a_i||c_i|)^{s/2} & \text{if $s\geq2$.}
     \end{array}\right.
\end{equation}

In the case of triangular matrices, the calculation of Lyapunov exponents of a self-affine measure $\mu$ with probability vector $(p_1,\dots,p_N)$ is much simpler. That is,
\begin{equation*}
\ly{ss}=\max\left\{-\sum_{i=1}^Np_i\log|a_i|,-\sum_{i=1}^Np_i\log|c_i|\right\}\text{ and }\ly{s}=\min\left\{-\sum_{i=1}^Np_i\log|a_i|,-\sum_{i=1}^Np_i\log|c_i|\right\}.
\end{equation*}

\begin{lemma}\label{ldomsplittri}
The set $\Alpha=\left\{A_1, A_2,\dots, A_N\right\}$ of contracting, non-singular $2\times2$ lower-triangular matrices in the form~\eqref{etrig}, satisfies the dominated splitting condition if
\[
\text{ either }|a_i|>|c_i|\text{ for every $i=1,\dots,N$ or }|a_i|<|c_i|\text{ for every $i=1,\dots,N$.}
\]
\end{lemma}

The proof of the lemma is straightforward by Lemma~\ref{ldomsplit1}.

In the case of triangular matrices the study of the dimension of self-affine set can be tracked back to study the dimension of some self-similar measure. In the first case of Lemma~\ref{ldomsplittri}, the projected measure in Theorem~\ref{tmain1} is a self-similar measure. In the second case, the measure $\nu_{ss}$ is a self-similar measure. We introduce here a condition, which guarantees according to the recent result of Hochman \cite[Theorem~1.1]{H} that the dimension of a self-similar measure is the quotient of the entropy and Lyapunov exponent. 

\begin{definition}
For a self-similar IFS $\phi=\left\{g_i(x)=\beta_ix+\gamma_i\right\}_{i=1}^N$ on the real line let
\begin{eqnarray*}& d(g_{\underline{i}},g_{\underline{j}}):=\left\{\begin{array}{cc}
                                            \infty & \text{if $\beta_{\underline{i}}\neq\beta_{\underline{j}}$} \\
                                            |g_{\underline{i}}(0)-g_{\underline{j}}(0)| & \text{if $\beta_{\underline{i}}=\beta_{\underline{j}}$}.
                                          \end{array}\right.\text{ and }\\
& \Delta_n:=\min\left\{d(g_{\underline{i}},g_{\underline{j}}):\underline{i}\neq\underline{j}\in\left\{1,\dots,N\right\}^n\right\}.
\end{eqnarray*} We say that the IFS $\phi$ satisfies the \underline{Hochman-condition} if
\[
\lim_{n\rightarrow\infty}-\frac{1}{n}\log\Delta_n<+\infty.
\]
\end{definition}

Hochman showed that the exceptional set of parameters, where the condition does not hold is small in sense of dimension, see \cite[Theorem~1.7, Theorem~1.8]{H}.

\begin{theorem}
Let $\Alpha=\left\{A_i\right\}_{i=1}^N$ be a finite set of triangular matrices of type~\eqref{etrig} and let\\ $\Phi=\left\{f_i(\xv)=A_i\xv+(t_i,q_i)\right\}_{i=1}^N$~be~an IFS on the plane. Suppose that
\begin{enumerate}
    \item\label{ctri11} $|a_i|>|c_i|$ for every $i=1,\dots,N$,
    \item\label{ctri12} $\Phi$ satisfies the strong separation condition,
    \item\label{ctri13} the self-similar IFS $\phi=\left\{g_i(x)=a_ix+t_i\right\}_{i=1}^N$ satisfies the Hochman-condition
\end{enumerate} then for every self-affine measure $\mu$
\begin{equation}\label{et11}
\dim_H\mu=\ldim\mu.
\end{equation}
Moreover,
\begin{equation}\label{et12}
\dim_H\Lambda=\dim_B\Lambda=\min\left\{s_1,s_2\right\},
\end{equation}
where $\Lambda$ is the attractor of $\Phi$ and $s_1$ and $s_2$ are the unique solutions of the equations
\begin{equation}\label{et13}
\sum_{i=1}^N|a_i|^{s_1}=1\text{ and }\sum_{i=1}^N|a_i||c_i|^{s_2-1}=1.
\end{equation}
\end{theorem}

\begin{proof}
Let $\nu=\left\{p_1,\dots,p_N\right\}^{\N}$ be an arbitrary Bernoulli-measure on $\Sigma^+$ and $\mu$ be the corresponding self-affine measure.

Condition~\eqref{ctri11} implies by Lemma~\ref{ldomsplittri} that the set $\Alpha$ of matrices satisfies dominated splitting and $e_{ss}(\ii)$, defined in Lemma~\ref{ldomsplit1}, is equal to the subspace parallel to the $y$-axis for every $\ii\in\Sigma^+$. Hence, the transversal measure $\mu_{\ii}^T\equiv\mu^T$, defined in \eqref{etransmeasure}, is a self-similar measure with the IFS $\phi=\left\{g_i(x)=a_ix+t_i\right\}_{i=1}^N$, namely
$$
\mu^T=\sum_{i=1}^Np_i\mu^T\circ g_i^{-1}.
$$
By condition~\eqref{ctri13} and \cite[Theorem~1.1]{H},
\[
\dim_H\mu^T=\min\left\{\frac{h_{\nu}}{\ly{s}},1\right\}.
\]
Thus, $\dim_H\mu^T=\min\left\{\ldim\mu,1\right\}$. By condition~\eqref{ctri12}, applying Theorem~\ref{tmain1} and \eqref{elyapHdimequal} we get \eqref{et11}.

To prove \eqref{et12}, first let us observe that condition~\eqref{ctri11} implies that the root of the subadditive pressure \eqref{etrigsap} is the minimum of the solutions of the equations \eqref{et13}. Then we get the upper bound by \cite[Proposition~5.1]{F}. The lower bound follows by choosing the measure $\nu$ according to the probability vector $\left\{|a_1|^{s_1},\dots,|a_N|^{s_1}\right\}$ or $\left\{|a_1||c_1|^{s_2-1},\dots,|a_N||c_N|^{s_2-1}\right\}$.
\end{proof}

Now we turn to the second case of Lemma~\ref{ldomsplittri}.

\begin{theorem}\label{ttri2}
Let $\Alpha=\left\{A_i\right\}_{i=1}^N$ be a finite set of triangular matrices of type~\eqref{etrig} and let\\ $\Phi=\left\{f_i(\xv)=A_i\xv+(t_i,q_i)\right\}_{i=1}^N$ be the IFS on the plane. Moreover, let $\nu$ be a left-shift invariant and ergodic Bernoulli-probability measure on $\Sigma^{+}$, and $\mu$ be the corresponding self-affine measure. Suppose that
\begin{enumerate}
    \item\label{ctri21} $|a_i|<|c_i|$ for every $i=1,\dots,N$,
    \item\label{ctri22} $\Phi$ satisfies the strong separation condition,
    \item\label{ctri23} the self-similar IFS $\phi=\left\{g_i(x)=\dfrac{a_i}{c_i}x-\dfrac{b_i}{c_i}\right\}_{i=1}^N$ satisfies the Hochman-condition,
    \item\label{ctri24} $\dfrac{h_{\nu}}{\ly{ss}-\ly{s}}\geq\min\left\{1,\dfrac{h_{\nu}}{\ly{s}}\right\}$
\end{enumerate} then
\begin{equation}\label{et21}
\dim_H\mu=\ldim\mu.
\end{equation}
Moreover, if condition~\eqref{ctri24} is replaced by the $1$-bunched property, i.e. $|a_i|\geq|c_i|^2$ then
\begin{equation}\label{et22}
\dim_H\Lambda=\dim_B\Lambda=\min\left\{s_1,s_2\right\},
\end{equation}
where $\Lambda$ is the attractor of $\Phi$ and $s_1$ and $s_2$ are the unique solutions of the equations
\begin{equation}\label{et23}
\sum_{i=1}^N|c_i|^{s_1}=1\text{ and }\sum_{i=1}^N|c_i||a_i|^{s_2-1}=1.
\end{equation}
\end{theorem}

We note if $\ldim\mu\leq1$ then condition~\eqref{ctri24} is basically the $1$-bunched property, defined in Theorem~\ref{tHL}. However, if $\ldim\mu>1$ then condition~\eqref{ctri24} is much relaxed and holds if $\ldim\mu$ is sufficiently large, for example if $h_{\nu}/\ly{ss}\geq1$.

\begin{lemma}\label{ltrig2}
Let $\Alpha=\left\{A_i\right\}_{i=1}^N$ be a finite set of matrices of type~\eqref{etrig} and let us suppose that $|a_i|<|c_i|$ for every $i=1,\dots,N$. Then the slopes of strong stable directions, defined in Lemma~\ref{ldomsplit1}, form a self-similar set of IFS $\phi=\left\{g_i(x)=\dfrac{a_i}{c_i}x-\dfrac{b_i}{c_i}\right\}_{i=1}^N$. In particular, for every $\ii=(i_0,i_1,\dots)\in\Sigma^+$ the subspace $e_{ss}(\ii)$ is parallel to the vector $v(\ii)=(1,\vartheta(\ii))^T$, where $$\vartheta(\ii)=-\sum_{n=0}^{\infty}\dfrac{b_{i_n}a_{i_{n-1}}\cdots a_{i_{0}}}{c_{i_n}c_{i_{n-1}}\cdots c_{i_0}}.$$
\end{lemma}

\begin{proof}
By simple algebraic calculations we have $$A_{i_0}v(\ii)=a_{i_0}v(\sigma\ii).$$
The statement follows by Lemma~\ref{ldomsplit2} and Lemma~\ref{ldomsplit3b}.
\end{proof}

An immediate consequence of Lemma~\ref{ltrig2} that for any Bernoulli measure on $\Sigma^+$
\begin{equation}\label{essdist}
\dim_H\nu_{ss}=\dim_H\nu\circ\vartheta^{-1},
\end{equation}
where $\nu_{ss}$ is defined in \eqref{edirections}.

\begin{proof}[Proof of Theorem~\ref{ttri2}]
First, let us observe that condition~\eqref{ctri23} with \eqref{essdist} and \cite[Theorem~1.1]{H} imply that
$$\dim_H\nu_{ss}=\min\left\{1,\dfrac{h_{\nu}}{\ly{ss}-\ly{s}}\right\}.$$
By Lemma~\ref{ldomsplittri}, condition~\eqref{ctri21} implies that the IFS~$\Phi$ satisfies dominated splitting, and together with conditions~\eqref{ctri22} and~\eqref{ctri24} by using Theorem~\ref{tmain1b}, \eqref{et21} follows.

To prove \eqref{et22}, first let us observe that condition~\eqref{ctri21} implies that the root of the subadditive pressure \eqref{etrigsap} is the minimum of the solutions of the equations \eqref{et23}. One can check that the $1$-bunched property implies that condition~\eqref{ctri24} holds for any self-similar measure. Hence, the lower bound follows by choosing the measure $\nu$ according to the probability vector $\left\{|c_1|^{s_1},\dots,|c_N|^{s_1}\right\}$ or $\left\{|c_1||a_1|^{s_2-1},\dots,|c_N||a_N|^{s_2-1}\right\}$. The upper bound follows by \cite[Proposition~5.1]{F}.
\end{proof}

\subsection{An example} Finally, we consider a concrete family of self-affine sets inspired by the example of Falconer and Miao, \cite[Figure~1]{FM}. Let $\Phi_c=\left\{f_1,\dots,f_6\right\}$ be a parameterized family of IFSs on the plane given by the functions
\begin{eqnarray*}
&f_1(\xv)=\left[\begin{array}{cc} \frac{1}{3} & 0 \\ 0 & c \\\end{array}\right]\xv+\left[\begin{array}{c} \frac{1}{3} \\ 0 \\\end{array}\right],
&f_2(\xv)=\left[\begin{array}{cc} \frac{1}{3} & 0 \\ 0 & c \\\end{array}\right]\xv+\left[\begin{array}{c} \frac{1}{3} \\ 1-c \\\end{array}\right],\\
&f_3(\xv)=\left[\begin{array}{cc} \frac{1}{3} & 0 \\ \frac{1}{2}-c & c \\\end{array}\right]\xv+\left[\begin{array}{c} 0 \\ \frac{1}{2} \\\end{array}\right],
&f_4(\xv)=\left[\begin{array}{cc} \frac{1}{3} & 0 \\ \frac{1}{2}-c & c \\\end{array}\right]\xv+\left[\begin{array}{c} \frac{2}{3} \\ 0 \\\end{array}\right],\\
&f_5(\xv)=\left[\begin{array}{cc} \frac{1}{3} & 0 \\ c-\frac{1}{2} & c \\\end{array}\right]\xv+\left[\begin{array}{c} 0 \\ \frac{1}{2}-c \\\end{array}\right],
&f_6(\xv)=\left[\begin{array}{cc} \frac{1}{3} & 0 \\ c-\frac{1}{2} & c \\\end{array}\right]\xv+\left[\begin{array}{c} \frac{2}{3} \\ 1-c \\\end{array}\right],
\end{eqnarray*}
where $0<c<1/2$. Let $\Lambda_c$ denote the attractor of $\Phi_c$, see Figure~\ref{fset}.

\begin{figure}[ht]
  \centering
  \includegraphics[width=60mm]{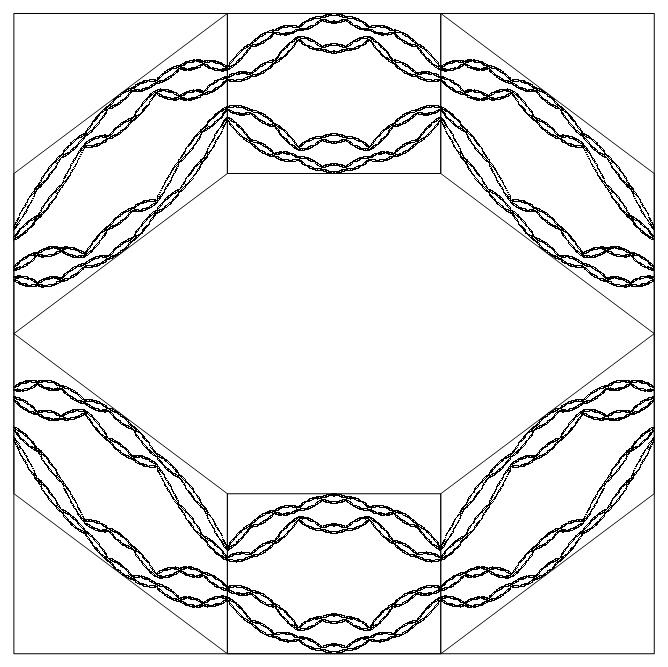} \includegraphics[width=60mm]{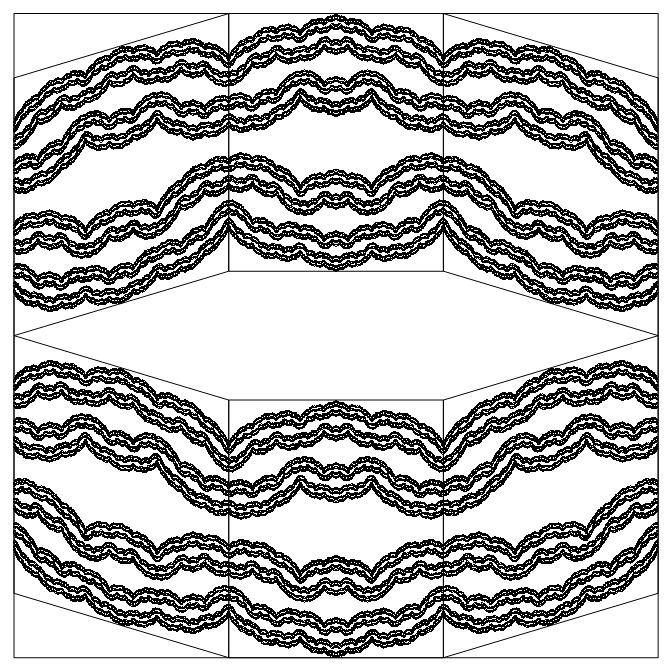}\\
  \caption{The attractors $\Lambda_c$ of IFSs $\Phi_c$ with parameters $c=0.25$ and $c=0.4$. The affine maps are those that map the unit square to the parallelograms shown.}\label{fset}
\end{figure}

\begin{theorem}
For every $0<c<\frac{1}{3}$
\begin{equation}\label{etex1}
\dim_H\Lambda_c=\dim_B\Lambda_c=1-\frac{\log2}{\log c},
\end{equation}
and there exists a set $\mathcal{C}\subseteq(\frac{1}{3},\frac{1}{2})$ such that $\dim_P\mathcal{C}=0$ and
\begin{equation}\label{etex2}
\dim_H\Lambda_c=\dim_B\Lambda_c=2+\frac{\log2c}{\log3}\text{ for every $c\in\left(\frac{1}{3},\frac{1}{2}\right)\backslash\mathcal{C}$.}
\end{equation}
\end{theorem}

The box dimension of $\Lambda_c$ is already known for every $c\in(0,\frac{1}{2})$ by \cite[Corollary~5]{F2}.

\begin{proof}
Let $\mathcal{S}=\left\{1,\dots,6\right\}$ denote the set of symbols and let $\widetilde{\mathcal{S}}^n:=\mathcal{S}^n\backslash\left\{4,6\right\}^n$.

Observe that the IFS $\Phi_c$ satisfies the open set condition but not the strong separation condition. However, the IFS $\widetilde{\Phi}^n_c$ given by $\widetilde{\Phi}^n_c=\left\{f_{\underline{i}}\right\}_{\underline{i}\in\widetilde{S}^n}$ satisfies SSC for every $n\geq1$ and $0<c<\frac{1}{2}$. Denote the attractor of $\widetilde{\Phi}^n_c$ by $\widetilde{\Lambda}_{n,c}$. For every $n\geq1$ let $\widetilde{\Sigma}_n=\left(\widetilde{S}^n\right)^{\N}$ be the symbolic space and $\nu^{(n)}$ let be the uniform Bernoulli measure on $\widetilde{\Sigma}_n$ and $\widetilde{\mu}_{n,c}$ the corresponding self-affine measure supported on $\widetilde{\Lambda}_{n,c}$.

First, let us consider the case $0<c<\frac{1}{3}$. Then by Lemma~\ref{ldomsplittri} the IFS $\widetilde{\Phi}^n_c$ satisfies dominated splitting and by Lemma~\ref{ldomsplit1} and Lemma~\ref{ldomsplit2}, the strong stable directions are parallel to the $y$-axis. Hence, the transversal measure $\widetilde{\mu}^T_{n,c}$, defined in \eqref{etransmeasure}, is a self-similar measure with uniform probabilities, satisfying SSC. Thus,
$$\dim_H\widetilde{\mu}^T_{n,c}=\frac{\log(3^n-1)}{\log3^n}.$$ Applying Theorem~\ref{tmain1}, we get
\begin{multline*}
\dim_H\Lambda\geq\lim_{n\to\infty}\dim_H\widetilde{\Lambda}_{n,c}\geq\\
\lim_{n\to\infty}\dim_H\widetilde{\mu}_{n,c}=\lim_{n\to\infty}\frac{\log(6^n-2^n)}{-\log c^n}+\left(1-\frac{\log3^n}{-\log c^n}\right)\frac{\log(3^n-1)}{\log3^n}=1-\frac{\log2}{\log c},
\end{multline*}
which proves \eqref{etex1}.

Now we turn to the case $\frac{1}{3}<c<\frac{1}{2}$. Lemma~\ref{ldomsplittri} implies that the IFS $\Phi$ satisfies again the dominated splitting condition, moreover, by Lemma~\ref{ldomsplit1} and Lemma~\ref{ltrig2}, the strong stable directions can be given by the IFS $\phi_c$ of similarities $$g_1(x)=g_2(x)=\frac{1}{3c}x,\ g_3(x)=g_4(x)=\frac{1}{3c}x+\frac{2c-1}{2c},\text{ and }g_5(x)=g_6(x)=\frac{1}{3c}x+\frac{1-2c}{2c}.$$ Thus by \eqref{essdist}, the distribution $\widetilde{\nu}_{n,c}^{ss}$ of strong stable directions of the IFS $\widetilde{\Phi}^n_c$ are given by the IFS $\widetilde{\phi}_c^n=\left\{g_{\underline{i}}\right\}_{\underline{i}\in\widetilde{S}^n}$ with the uniform Bernoulli measure on $\widetilde{\Sigma}_n$. Applying \cite[Theorem~1.8]{H} we get that for every $n\geq1$ there exists a set $\mathcal{C}_n$ with $\dim_P\mathcal{C}_n=0$ such that
\[
\dim_H\widetilde{\nu}_{n,c}^{ss}=\min\left\{1,\dfrac{-2^n\frac{2^n-1}{6^n-2^n}\log\frac{2^n-1}{6^n-2^n}-(3^n-2^n)\frac{1}{3^n-1}\log\frac{1}{3^n-1}}{\log(3c)^n}\right\}\text{ for every }c\in\left(\frac{1}{3},\frac{1}{2}\right)\backslash\mathcal{C}_n.
\]
For sufficiently large $n$, we apply Theorem~\ref{tmain1b}, and therefore
\begin{multline*}
\dim_H\Lambda\geq\lim_{n\to\infty}\dim_H\widetilde{\Lambda}_{n,c}\geq\\
\lim_{n\to\infty}\dim_H\widetilde{\mu}_{n,c}=\lim_{n\to\infty}1+\frac{\log(6^n-2^n)-(-\log  c^n)}{\log 3^n}=2+\frac{\log2c}{\log3}\text{ for every }c\in\left(\frac{1}{3},\frac{1}{2}\right)\backslash\mathcal{C},
\end{multline*}
where $\mathcal{C}=\bigcup_{n=1}^{\infty}\mathcal{C}_n$, which proves \eqref{etex2}.
\end{proof}

\subsection{An applications for Theorem~\ref{tmain1c}}\label{ssthm1c}

Finally, we show an application for Theorem~\ref{tmain1c}. It is non-trivial to check whether condition Theorem~\ref{tmain1c}\eqref{cmain1c} holds. We replace it with a condition that can be checked much easier similarly to Theorem~\ref{tHL}.

\begin{theorem}\label{tmain1capp}
	Let $\Alpha=\left\{A_1, A_2,\dots, A_N\right\}$ be a finite set of contracting, non-singular $2\times2$ matrices, and let $\Phi=\left\{f_i(\xv)=A_i\xv+\tv_i\right\}_{i=1}^N$ be an iterated function system on the plane with affine mappings. Let $\nu$ be a left-shift invariant and ergodic Bernoulli-probability measure on $\Sigma^{+}$, and $\mu$ be the corresponding self-affine measure. Assume that
	\begin{enumerate}
		\item\label{c1c1} $\Alpha$ satisfies the dominated splitting condition,
		\item\label{c1c2} $\Alpha$ satisfies the backward non-overlapping condition, i.e. there exists a backward invariant multicone $M$ such that $A^{-1}_i(M^{o})\subseteq M^{o}$ and $A^{-1}_i(M^{o})\cap A_j^{-1}(M^{o})=\emptyset$ for every $i\neq j$,
		\item\label{c1c3} $\Phi$ satisfies the strong separation condition,
		\item\label{c1c4} $\dfrac{h_{\nu}}{\ly{ss}-\ly{s}}+2\dfrac{h_{\nu}}{\ly{ss}}>2$.
	\end{enumerate}
	Then
	\begin{equation*}
	\dim_H\mu=\ldim\mu. 
	\end{equation*}
\end{theorem}

By Theorem~\ref{tmain1}, we get the trivial lower bound $\dim_H\mu\geq\frac{h_{\nu}}{\ly{ss}}$. Unfortunately, if the backward non-overlapping condition holds then $\frac{h_{\nu}}{\ly{ss}}+\frac{h_{\nu}}{\ly{ss}-\ly{s}}<2\frac{h_{\nu}}{\ly{ss}-\ly{s}}<2$. Therefore, we need to improve the lower bound of $\dim_H\mu$.

\begin{lemma}\label{lto1c}
	Let us assume that conditions \eqref{c1c1}-\eqref{c1c3} of Theorem~\ref{tmain1capp} hold. Then
	$$
	\dim_H\mu\geq\min\left\{2\dfrac{h_{\nu}}{\ly{ss}},\dfrac{h_{\nu}}{\ly{s}}\right\}.
	$$
\end{lemma}

\begin{proof}
	Let us define a sequence $\left\{x_n\right\}_{n=0}^{\infty}$ as follows, let $x_0=\dfrac{h_{\nu}}{\ly{ss}}$ and $x_n=f(x_{n-1})$ for $n\geq1$, where
	$$
	f(x)=\dfrac{h_{\nu}}{\ly{ss}}+\left(1-\dfrac{\ly{s}}{\ly{ss}}\right)\min\left\{\dfrac{h_{\nu}}{\ly{ss}-\ly{s}},x\right\}.
	$$
	It is easy to see that the sequence $\left\{x_n\right\}_{n=0}^{\infty}$ converges to $\min\left\{2\dfrac{h_{\nu}}{\ly{ss}},\dfrac{h_{\nu}}{\ly{s}}\right\}$, which is the fixed point of $x\mapsto f(x)$. By applying Theorem~\ref{tmain1} and Lemma~\ref{lmarstrand}, one can show by induction that $\dim_H\mu\geq x_n$ for every $n\geq0$, as required.
\end{proof}

\begin{proof}[Proof of Theorem~\ref{tmain1capp}]
	If $\dfrac{h_{\nu}}{\ly{ss}-\ly{s}}\geq\dfrac{h_{\nu}}{\ly{s}}$ then we may apply Theorem~\ref{tHL} and the statement holds. Thus, without loss of generality we may assume that $\dfrac{h_{\nu}}{\ly{ss}-\ly{s}}<\dfrac{h_{\nu}}{\ly{s}}$, or equivalently $2\ly{s}<\ly{ss}$. Then by Lemma~\ref{lto1c} we get that $\dim_H\mu\geq2\dfrac{h_{\nu}}{\ly{ss}}$. Therefore,
	$$
	\dim_H\nu_{ss}+\dim_H\mu\geq\dfrac{h_{\nu}}{\ly{ss}-\ly{s}}+2\dfrac{h_{\nu}}{\ly{ss}}>2.
	$$
	The statement follows by Theorem~\ref{tmain1c}.
\end{proof}

\begin{figure}[ht]
	\centering
	\includegraphics[width=65mm]{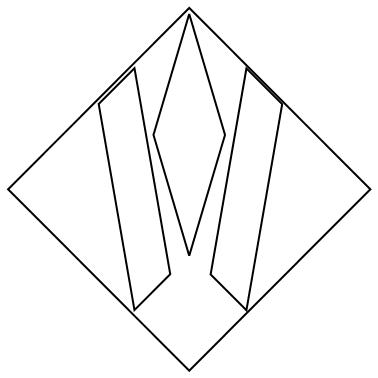} 
	\includegraphics[width=65mm]{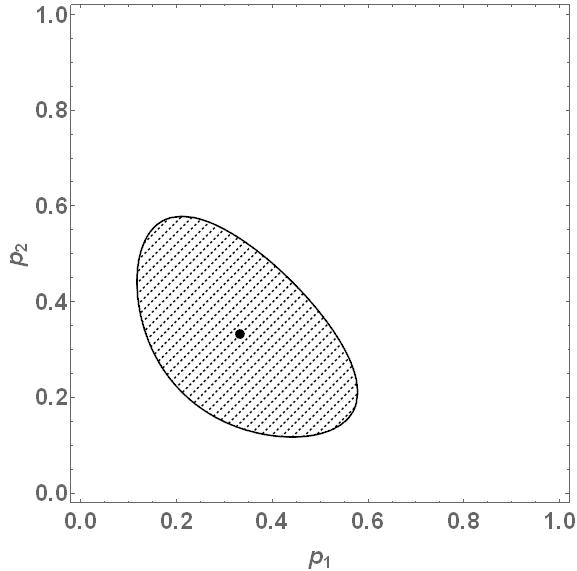}\\
	\caption{The images of parallelogram $\mathcal{O}$ by the affine maps $f_1,f_2,f_3$ and the region of probability vectors $(p_1,p_2,1-p_1-p_2)$, where Theorem~\ref{tmain1capp}\eqref{c1c4} holds. }\label{fexample}
\end{figure}

 Let $\Phi=\left\{f_1,f_2,f_3\right\}$ be an IFSs on the plane given by the functions
$$
f_1(\xv)=\left[\begin{array}{cc} \frac{16}{81} & 0 \\[0.3em] -\frac{2}{3} & \frac{2}{3} \\\end{array}\right]\xv+\left[\begin{array}{c} \frac{19}{54} \\[0.3em] \frac{47}{100} \\\end{array}\right],\ f_2(\xv)=\left[\begin{array}{cc} \frac{16}{81} & 0 \\[0.3em] 0 & \frac{2}{3} \\\end{array}\right]\xv+\left[\begin{array}{c} \frac{1235}{2187} \\[0.3em] \frac{3}{10} \\\end{array}\right]\text{ and }f_3(\xv)=\left[\begin{array}{cc} \frac{16}{81} & 0 \\[0.3em] \frac{2}{3} & \frac{2}{3} \\\end{array}\right]\xv+\left[\begin{array}{c} \frac{1721}{2187} \\[0.3em] -\frac{38}{81} \\\end{array}\right],
$$
and let us denote the attractor of $\Phi$ by $\Lambda$.

By Lemma~\ref{ldomsplittri}, it satisfies the dominated splitting condition, moreover, it satisfies the strong separation condition with the parallelogram $\mathcal{O}$ formed by vertices $(0,0), \left(\dfrac{19}{27},1\right), \left(\dfrac{38}{27},0\right), \left(\dfrac{19}{27},-1\right)$, see Figure~\ref{fexample}. By Lemma~\ref{ltrig2}, the strong stable directions are formed by the self-similar IFS\linebreak $\left\{x\mapsto\dfrac{8}{27}x+1,x\mapsto\dfrac{8}{27}x,x\mapsto\dfrac{8}{27}x-1\right\}$, which satisfies the strong separation condition. 

Let us consider, the  uniformly distributed Bernoulli measure on $\left\{1,2,3\right\}^{\N}$, and the corresponding Bernoulli measure $\mu=\frac{1}{3}(f_1)_*\mu+\frac{1}{3}(f_2)_*\mu+\frac{1}{3}(f_3)_*\mu$. Then
$$
\dfrac{h_{\nu}}{\ly{ss}-\ly{s}}+2\dfrac{h_{\nu}}{\ly{ss}}=\dfrac{\log3}{\log27/8}+\dfrac{2\log3}{\log81/16}>2.
$$
So we may apply Theorem~\ref{tmain1capp} and we get
$$
\dim_H\mu=\dim_H\Lambda=\dim_B\Lambda=1+\dfrac{\log3-\log3/2}{\log81/16}\approx1.4273.
$$
For the complete region of probability vectors, where condition Theorem~\ref{tmain1capp}\eqref{c1c4} holds, see Figure~\ref{fexample}. 

For other examples, see Falconer and Kempton~\cite{FK}.

\subsection*{Acknowledgement}
	The author is extremely grateful to the anonymous referee for the useful comments and suggestions. Especially, to point out Theorem~\ref{tmain1c}.

\end{document}